\documentclass[11pt,twoside, leqno]{article}

\usepackage{amssymb}
\usepackage{amsmath}
\usepackage{mathrsfs}
\usepackage{amsthm}
\usepackage{txfonts}
\usepackage{color}

\usepackage{indentfirst}

\allowdisplaybreaks

\def\ls{\lesssim}

\def\fz{\infty}

\def\red{\color{red}}

\renewcommand{\r}{\right}
\newcommand{\lf}{\left}

\def\ls{\lesssim}

\def\supp{{\mathop\mathrm{\,supp\,}}}

\def\rr{{\mathbb R}}

\def\rn{{{\rr}^n}}

\newcommand{\wz}{\widetilde}

\newcommand{\cm}{{\mathcal M}}

\def\az{\alpha}
\def\lz{\lambda}

\def\dz{\delta}

\def\bz{\beta}

\def\fai{\varphi}
\def\gz{{\gamma}}

\def\tz{\theta}
\def\sz{\sigma}

\def\wz{\widetilde}

\def\ls{\lesssim}

\def\boz{\Omega}

\def\uc{{\varepsilon}}

\def\esup{\mathop\mathrm{\,ess\,sup\,}}

\def\hs{\hspace{0.3cm}}

\newtheorem{theorem}{Theorem}[section]
\newtheorem{lemma}[theorem]{Lemma}
\newtheorem{corollary}[theorem]{Corollary}
\newtheorem{assumption}[theorem]{Assumption}

\theoremstyle{definition}
\newtheorem{remark}[theorem]{Remark}
\newtheorem{definition}[theorem]{Definition}
\def\supp{{\mathop\mathrm{\,supp\,}}}
\def\diam{{\mathop\mathrm{diam\,}}}
\def\dist{{\mathop\mathrm{\,dist\,}}}
\def\loc{{\mathop\mathrm{loc}}}

\numberwithin{equation}{section}

\begin{document}

\title{\Large\bf Global Gradient Estimates for Dirichlet Problems of Elliptic Operators
with a BMO Anti-Symmetric Part
\footnotetext{\hspace{-0.35cm} 2020 {\it Mathematics Subject
Classification}. {Primary 35J25; Secondary 35J15, 42B35, 42B37.}
\endgraf{\it Key words and phrases}. elliptic operator, Dirichlet problem, NTA domain,
quasi-convex domain,  weak reverse H\"older inequality,
gradient estimate, Muckenhoupt weight.
\endgraf This work is supported by the National Natural Science Foundation
of China (Grant Nos. 11871254, 12071431, 11971058, 12071197 and 11871100)
and the National Key Research and Development Program of China
(Grant No.\ 2020YFA0712900).}}
\author{Sibei Yang, Dachun Yang
 and Wen Yuan\,\footnote{Corresponding author,
E-mail: \texttt{wenyuan@bnu.edu.cn}/{\red November 17, 2021}/Final version.}}
\date{ }
\maketitle

\vspace{-0.8cm}

\begin{center}
\begin{minipage}{13.5cm}\small
{{\bf Abstract.} Let $n\ge2$ and $\boz\subset\mathbb{R}^n$
be a bounded NTA domain. In this article, the authors investigate (weighted) global gradient estimates for
Dirichlet boundary value problems of second order elliptic equations of divergence form
with an elliptic symmetric part and a BMO anti-symmetric part in $\Omega$.
More precisely, for any given $p\in(2,\infty)$, the authors prove that a weak reverse H\"older
inequality with exponent $p$ implies the global $W^{1,p}$ estimate and
the global weighted $W^{1,q}$ estimate, with $q\in[2,p]$ and some Muckenhoupt weights,
of solutions to Dirichlet boundary value problems.
As applications, the authors establish some global gradient estimates for solutions to Dirichlet
boundary value problems of second order elliptic equations of divergence form with small
$\mathrm{BMO}$ symmetric part and small $\mathrm{BMO}$ anti-symmetric part, respectively,
on bounded Lipschitz domains, quasi-convex domains, Reifenberg flat domains, $C^1$
domains, or (semi-)convex domains, in weighted Lebesgue spaces.
Furthermore, as further applications, the authors obtain the global gradient estimate,
respectively, in (weighted) Lorentz spaces, (Lorentz--)Morrey spaces,
(Musielak--)Orlicz spaces, and variable Lebesgue spaces. Even on global gradient
estimates in Lebesgue spaces, the results obtained in this article improve the
known results via weakening the assumption on the coefficient matrix.}
\end{minipage}
\end{center}

\vspace{0.1cm}

\section{Introduction\label{s1}}

 It is well known that the research of global regularity estimates in various function spaces
for (non-)linear elliptic equations (or systems) in non-smooth domains
is one of the most interesting and important topics in partial differential equations
(see, for instance, \cite{amp18,aq02,b05,bw04,d96,dk10,dk18,g12,jk95,sh05a}
for the linear case, and \cite{ap15,b18,bd18,byz08,cp98,mp12,mp11,yzz19,zz14} for the non-linear case).
Moreover, the global regularity estimates for solutions to
elliptic boundary problems depend not only on the structure of equations,
the integrability of the right-hand side datum, and the properties of coefficients appearing
in equations, but also on the smooth property or
the geometric property of the boundary of domains
(see, for instance, \cite{amp18,aq02,bw04,dk12,dk18,g12,jlw10,k07,mp12,ycyy20}).

Motivated by \cite{g12,lp19,sssz12,sh05a,ycyy20}, in this article, we study
global gradient estimates in various function spaces for
Dirichlet boundary value problems of second order elliptic equations of divergence form
with an elliptic symmetric part and a BMO anti-symmetric part on non-smooth domains of $\rn$.
More precisely, let $n\ge2$ and $\boz\subset\mathbb{R}^n$
be a bounded non-tangentially accessible domain (for short,
NTA domain). For any given $p\in(2,\infty)$, using a real-variable argument,
we show that a weak reverse H\"older
inequality with exponent $p$ implies the global $W^{1,p}$ estimate and
the global weighted $W^{1,q}$ estimate, with $q\in[2,p]$ and some Muckenhoupt weights,
of solutions to Dirichlet boundary value problems.
As applications, we obtain some global gradient estimates for solutions to Dirichlet
boundary value problems of second order elliptic equations of divergence form with small
$\mathrm{BMO}$ symmetric part and small $\mathrm{BMO}$ anti-symmetric part, respectively,
on bounded Lipschitz domains, quasi-convex domains, Reifenberg flat domains, $C^1$
domains, or (semi-)convex domains, in the scale of weighted Lebesgue spaces.
Applying these weighted global estimates and some
technique from harmonic analysis, such as properties of Muckenhoupt weights, the interpolation theorem of
operators, and the extrapolation theorem, we further establish the global gradient estimate,
respectively, in (weighted) Lorentz spaces, (Lorentz--)Morrey spaces, (Musielak--)Orlicz
spaces, and variable Lebesgue spaces. Even on global gradient
estimates in Lebesgue spaces, the results obtained in this article improve the
corresponding results in \cite{amp18,bw04,jlw10,sh05a}
via weakening the assumption on the coefficient matrix.

To describe the main results of this article,
we first recall the notions of the Muckenhoupt weight class
and the reverse H\"older class (see, for instance, \cite{am07,g14,St93}).

\begin{definition}\label{d1.1}
Let $q\in[1,\fz)$. A non-negative and locally integrable function $\omega$ on $\rn$
is said to belong to the \emph{Muckenhoupt weight class} $A_q(\rn)$, denoted by $\omega\in A_q(\rn)$,
if, when $q\in(1,\fz)$,
\begin{equation*}
[\omega]_{A_q(\rn)}:=\sup_{B\subset\rn}\lf[\frac{1}{|B|}\int_B
\omega(x)\,dx\r]\lf\{\frac{1}{|B|}\int_B
[\omega(x)]^{-\frac{1}{q-1}}\,dx\r\}^{q-1}<\fz,
\end{equation*}
or
\begin{equation*}
[\omega]_{A_1(\rn)}:=\sup_{B\subset\rn}\lf[\frac{1}{|B|}\int_B \omega(x)\,dx\r]
\lf\{\esup_{y\in B}[\omega(y)]^{-1}\r\}<\fz,
\end{equation*}
where the suprema are taken over all balls $B$ of $\rn$.

Let $r\in(1,\fz]$. A non-negative and locally integrable function $\omega$ on $\rn$
is said to belong to the \emph{reverse H\"older class} $RH_r(\rn)$,
denoted by $\omega\in RH_r(\rn)$, if, when $r\in(1,\fz)$,
\begin{align*}
[\omega]_{RH_r(\rn)}:=\sup_{B\subset\rn}\lf\{\frac{1}
{|B|}\int_B [\omega(x)]^r\,dx\r\}^{\frac1r}\lf[\frac{1}{|B|}\int_B
\omega(x)\,dx\r]^{-1}<\fz,
\end{align*}
or
\begin{equation*}
[\omega]_{RH_\fz(\rn)}:=\sup_{B\subset\rn}\lf[\esup_{y\in
B}\omega(y)\r]\lf[\frac{1}{|B|}\int_B\omega(x)\,dx\r]^{-1} <\fz,
\end{equation*}
where the suprema are taken over all balls $B$ of $\rn$.
\end{definition}

Furthermore, we recall the definition of the so-called NTA domain introduced by Jerison and Kenig
\cite{jk82} (see also \cite{kt97,t97}) as follows.
We begin with recalling several notions. For any given $x\in\rn$ and measurable subset $E\subset\rn$,
let $\dist(x,E):=\inf\{|x-y|:\ y\in E\}$. Meanwhile, for any measurable subsets $E,\ F\subset\rn$,
let $\dist(E,F):=\inf\{|x-y|:\ x\in E,\,y\in F\}$ and $\diam(E):=\sup\{|x-y|:\ x,\,y\in E\}$.
Moreover, for any given $x\in\rn$ and $r\in(0,\fz)$, let $B(x,r):=\{y\in\rn:\ |y-x|<r\}$.

\begin{definition}\label{d1.2}
Let $n\ge2$, $\boz\subset\rn$ be a \emph{domain}
which means that $\boz$ is a connected open set, and $\boz^\complement:=\rn\backslash\boz$.
Denote by $\partial\boz$ the \emph{boundary} of $\boz$.
\begin{itemize}
\item[{\rm(i)}] Then the domain $\boz$ is said to satisfy the \emph{interior}
[resp., the \emph{exterior}] \emph{corkscrew condition}
 if there exist positive constants $R\in(0,\fz)$ and $\sz\in(0,1)$ such that,
for any $x\in\partial\boz$ and $r\in(0,R)$, there exists a point $x_0\in\boz$
[resp., $x_0\in\boz^\complement$], depending on $x$,
such that $B(x_0,\sz r)\subset\boz\cap B(x,r)$ [resp., $B(x_0,\sz r)\subset\boz^\complement\cap B(x,r)$].
\item[{\rm(ii)}] The domain $\boz$ is said to satisfy the \emph{Harnack chain condition} if there exist constants
$m_1\in(1,\fz)$ and $m_2\in(0,\fz)$ such that, for any $x_1,\,x_2\in\boz$ satisfying
$$M:=\frac{|x_1-x_2|}{\min\{\dist(x_1,\partial\boz),\dist(x_2,\partial\boz)\}}>1,
$$
there exists a chain $\{B_i\}_{i=1}^N$ of open Harnack balls,
$B_i\subset\boz$ for any $i\in\{1,\,\ldots,\,N\}$, that connects $x_1$ to $x_2$; namely,
$x_1\in B_1$, $x_2\in B_N$, $B_i\cap B_{i+1}\neq\emptyset$ for any $i\in\{1,\,\ldots,\,N-1\}$, and, for any
$i\in\{1,\,\ldots,\,N\}$,
$$m_1^{-1}\diam(B_i)\le\dist(B_i,\partial\boz)\le m_1\diam(B_i),
$$
where the integer $N$ satisfies $N\le m_2\log_2 M$.
\item[{\rm(iii)}] The domain $\boz$ is called a \emph{non-tangentially accessible domain} (for short,
NTA \emph{domain}) if $\boz$ satisfies the interior
and the exterior corkscrew conditions, and the Harnack chain condition.
\end{itemize}
\end{definition}

We point out that NTA domains include Lipschitz domains, Zygmund domains, quasi-spheres,
and some Reifenberg flat domains as special examples (see, for instance, \cite{jk82,kt97,t97}).

Let $n\ge2$ and $\boz$ be a bounded NTA domain in $\rn$.
Assume that $p\in[1,\fz)$ and $\omega\in A_q(\rn)$ with some $q\in[1,\fz)$.
Recall that the \emph{weighted Lebesgue space} $L^p_\omega(\Omega)$ is defined by setting
\begin{align}\label{1.1}
L^p_\omega(\Omega):=\lf\{f\ \text{is measurable on}\ \Omega: \
\|f\|_{L^p_\omega(\Omega)}:=\lf[\int_{\boz}
|f(x)|^p\omega(x)\,dx\r]^{\frac1p}<\fz\r\}.
\end{align}
Moreover, let
\begin{equation}\label{1.2}
L^p_\omega(\boz;\rn):=\lf\{\mathbf{f}:=(f_1,\,\ldots,\,f_n):\ \text{for any}
\ i\in\{1,\,\ldots,\,n\},\ f_i\in L^p_\omega(\boz)\r\}
\end{equation}
and
$$\|\mathbf{f}\|_{L^p_\omega(\boz;\rn)}:
=\sum_{i=1}^n\|f_i\|_{L^p_\omega(\boz)}.
$$
Denote by $W^{1,p}_\omega(\boz)$ the \emph{weighted Sobolev space on $\boz$} equipped
with the \emph{norm}
$$\|f\|_{W^{1,p}_\omega(\boz)}:=\|f\|_{L^p_\omega(\boz)}+\|\nabla f\|_{L^p_\omega
(\boz;\rn)},$$
where $\nabla f$ denotes the \emph{distributional
gradient} of $f$. Furthermore, $W^{1,p}_{0,\,\omega}(\boz)$ is defined to be
the \emph{closure} of $C^{\fz}_{\mathrm{c}} (\boz)$ in $W^{1,p}_\omega(\boz)$, where
$C^{\fz}_{\mathrm{c}}(\boz)$ denotes the set of all \emph{infinitely differentiable functions on
$\boz$ with compact support contained in $\boz$}.
In particular, when $\omega\equiv1$, the weighted spaces $L^p_\omega(\boz)$,
$W^{1,p}_{\omega}(\boz)$, and $W^{1,p}_{0,\,\omega}(\boz)$ are denoted simply, respectively,
by $L^p(\boz)$, $W^{1,p}(\boz)$, and $W^{1,p}_0(\boz)$, which are
just, respectively, the classical Lebesgue space
and the classical Sobolev spaces.

Let $\boz$ be a domain of $\rn$. Denote by $L^1_{\loc}(\boz)$ the \emph{set of all locally integrable functions on $\boz$}.

\begin{definition}\label{d1.3}
Let $\boz\subset\rn$ be a domain and $f\in L^1_\loc(\boz)$.
Then $f$ is said to belong to the \emph{space} $\mathrm{BMO}(\boz)$ if
$$\|f\|_{\mathrm{BMO}(\boz)}:=\sup_{B\subset\boz}\frac{1}{|B|}\int_{B}|f(x)-f_B|\,dx<\fz,
$$
where the supremum is taken over all balls $B\subset\boz$ and $f_B:=\frac{1}{|B|}\int_B f(y)\,dy$.
\end{definition}

For any given $x\in\boz$, let $a(x):=\{a_{ij}(x)\}_{i,j=1}^n$ denote
an $n\times n$ matrix with real-valued, bounded and measurable entries.
Then $a$ is said to satisfy the \emph{uniform ellipticity condition}
if there exists a positive constant $\mu_0\in(0,1]$ such that,
for any $x\in\boz$ and $\xi:=(\xi_1,\,\ldots,\,\xi_n)\in\rn$,
\begin{equation}\label{1.3}
\mu_0|\xi|^2\le\sum_{i,j=1}^na_{ij}(x)\xi_i\xi_j\le \mu_0^{-1}|\xi|^2.
\end{equation}

Recall that the matrix $b:=\{b_{ij}\}_{i,j=1}^n$ is said to be \emph{anti-symmetric}
if $b_{ij}=-b_{ji}$ for any $i,\,j\in\{1,\,\ldots,\,n\}$.
Throughout this article, we \emph{always assume} that the matrix
$A$ satisfies the following assumption.

\begin{assumption}\label{a1}
Let $\boz\subset\rn$ be a domain. Assume that, for any given $x\in\boz$, $A(x)$ is an $n\times n$ matrix satisfying that
$A(x)=a(x)+b(x)$, where the matrix
$a(x):=\{a_{ij}(x)\}_{i,j=1}^n$ is real-valued, symmetric, and measurable, and satisfies
the uniform ellipticity condition \eqref{1.3}, and the matrix
$b(x):=\{b_{ij}(x)\}_{i,j=1}^n$ is real-valued, anti-symmetric, and measurable, and satisfies
$b_{ij}\in\mathrm{BMO}(\boz)$ for any $i,\,j\in\{1,\,\ldots,\,n\}$.
\end{assumption}

\begin{remark}\label{r1.1}
Let $\boz\subset\rn$ be a domain and $A:=a+b$ satisfy Assumption \ref{a1}.
\begin{itemize}
\item[\rm(i)] By the assumption that $a$ satisfies \eqref{1.3}, we conclude that
$a\in L^\fz(\boz;\rr^{n^2})$, which, together with the facts that $L^\fz(\boz)\subset\mathrm{BMO}(\boz)$
and $b\in\mathrm{BMO}(\boz;\rr^{n^2})$, further implies that $A\in \mathrm{BMO}(\boz;\rr^{n^2})$.

\item[\rm(ii)] Via replacing $\boz$ by $\rn$ in Definition \ref{d1.3},
we obtain the definition of the space $\mathrm{BMO}(\rn)$.
Jones \cite{j80} proved that any given function $f\in\mathrm{BMO}(\boz)$
admits an extension to some $\wz{f}\in\mathrm{BMO}(\rn)$ if and only if the domain $\boz$
is a \emph{uniform domain} (namely, the domain satisfying the interior corkscrew condition
and the Harnack chain condition). Thus, if $\boz$ is an NTA domain, then, for any given $f\in\mathrm{BMO}(\boz)$,
there exists an $\wz{f}\in\mathrm{BMO}(\rn)$ such that
$$\wz{f}\Big|_\boz=f\quad\text{and}\quad \lf\|\wz{f}\r\|_{\mathrm{BMO}(\rn)}\le C
\|f\|_{\mathrm{BMO}(\boz)},
$$
where $C$ is a positive constant depending only on $\boz$ and $n$.

\item[\rm(iii)] By the assumptions that $a$ satisfies \eqref{1.3} and $b$ is anti-symmetric,
we conclude that, for any $x\in\boz$ and $\xi\in\rn$,
$$\lf(A(x)\xi\r)\cdot\xi=(a(x)\xi)\cdot\xi\ge\mu_0|\xi|^2.
$$
\end{itemize}
\end{remark}

Let $\boz\subset\rn$ be a bounded domain and the matrix $A$ satisfy Assumption
\ref{a1}. Assume that $p\in(1,\fz)$, $\omega\in A_q(\rn)$ with some $q\in[1,\fz)$,
and $\mathbf{f}\in L^p_\omega(\boz;\rn)$. Then a function $u$ is called a \emph{weak solution} of the
following \emph{weighted Dirichlet boundary value problem}
\begin{equation}\label{1.4}
\begin{cases}
-\mathrm{div}(A\nabla u)=\mathrm{div}(\mathbf{f})\ \ &\text{in}\ \ \boz,\\
u=0 \ \ &\text{on}\ \ \partial\boz
\end{cases}
\hspace{4cm}(D)_{p,\,\omega}
\end{equation}
if $u\in W^{1,p}_{0,\,\omega}(\boz)$ and, for any $\varphi\in C^\fz_\mathrm{c}(\boz)$,
\begin{equation}\label{1.5}
\int_{\boz}A(x)\nabla u(x)\cdot\nabla\varphi(x)\,dx=-\int_\boz\mathbf{f}(x)\cdot\nabla\varphi(x)\,dx.
\end{equation}
In particular, when $\omega\equiv1$,
the weighted Dirichlet problem $(D)_{p,\,\omega}$ is just the Dirichlet problem $(D)_{p}$.
The weighted Dirichlet problem $(D)_{p,\,\omega}$ is said to
be \emph{uniquely solvable} if, for any given $\mathbf{f}\in L^p_\omega(\boz;\rn)$,
there exists a \emph{unique} $u\in W^{1,p}_{0,\,\omega}(\boz)$
such that \eqref{1.5} holds true for any $\varphi\in C^\fz_{\mathrm{c}}(\boz)$.

Let $L:=-\mathrm{div}(A\nabla)$ with the matrix $A$ satisfying Assumption \ref{a1}.
The elliptic operator $L$ naturally arises in the study of the elliptic equation of the form
\begin{equation}\label{1.6}
-\Delta u+\mathbf{c}\cdot\nabla u=f
\end{equation}
(see, for instance, \cite{mv06,sssz12}) and the parabolic  equation
$$\frac{\partial u}{\partial t}-\Delta u+\mathbf{c}\cdot\nabla u=f,
$$
where the drift term $\mathbf{c}$ satisfies $\mathrm{div}(\mathbf{c})=0$.
By $\mathrm{div}(\mathbf{c})=0$, we know that $\mathbf{c}=\mathrm{div}(b)$
for some anti-symmetric tensor $b:=\{b_{ij}\}_{i,j=1}^n$. Therefore, the equation
\eqref{1.6} becomes
$$-\mathrm{div}(I-b)\nabla u=f,
$$
where $I$ denotes the unit matrix on $\rn$. In particular, Seregin et al. \cite{sssz12}
discovered that the well-known Moser iteration works for such an elliptic operator
$L$. Via the Moser iteration,  Seregin et al. \cite{sssz12} proved
the Liouville theorem and the Harnack inequality for solutions
to the equation \eqref{1.6} or its parabolic case. Moreover, Li and Pipher
\cite{lp19} studied the boundary behavior of solutions of
the equation $Lu=0$ in NTA domains. Furthermore,
Dong and Phan \cite{dp21} investigated the mixed-norm Sobolev estimate for solutions to non-stationary
Stokes systems with coefficients having unbounded anti-symmetric part in cylindrical domains.

\begin{remark}\label{r1.2}
Let $\boz\subset\rn$ be a bounded NTA domain and the matrix $A:=a+b$ satisfy Assumption \ref{a1}.
For any $u,\,v\in W^{1,2}_0(\boz)$, let
$$B[u,v]:=\int_\boz A(x)\nabla u(x)\cdot\nabla v(x)\,dx.
$$
From Remark \ref{r1.1}(iii), it follows that, for any $u\in W^{1,2}_0(\boz)$,
$$B[u,u]\ge\mu_0\|\nabla u\|_{L^2(\boz;\rn)}^2.$$
Moreover, it was showed in \cite[(2.11)]{lp19}, via using the div-curl lemma, that,
for any $u,\,v\in W^{1,2}_0(\boz)$,
$$|B[u,v]|\le C\|\nabla u\|_{L^2(\boz;\rn)}\|\nabla v\|_{L^2(\boz;\rn)},$$
where $C$ is a positive constant depending only on $A$ and $\boz$.
Thus, by the Lax--Milgram theorem (see, for instance, \cite[Theorem 5.8]{gt01}),
we conclude that the Dirichlet problem $(D)_2$ is uniquely solvable
and, for any given $\mathbf{f}\in L^2(\boz;\rn)$, the weak solution $u\in W^{1,2}_0(\boz)$
of $(D)_2$ satisfies that
$$\|\nabla u\|_{L^2(\boz;\rn)}\le \mu_0^{-1}\|\mathbf{f}\|_{L^2(\boz;\rn)},
$$
where $\mu_0$ is as in \eqref{1.3}.

Moreover, via an example given by Meyers \cite[Section 5]{m63} (see also \cite[p.\,1285]{bw04}),
we know that, for any $p\in(1,\fz)$ with $p\neq2$,
the Dirichlet problem $(D)_p$ may not be uniquely solvable,
even when the domain $\boz$ is smooth.
\end{remark}

Let $n\ge2$, $\boz\subset\rn$ be a bounded NTA domain, and the matrix $A$
satisfy Assumption \ref{a1}. Assume further that $A$ satisfies
the $(\dz,R)$-BMO condition (see Definition \ref{d2.1} below)
or $A$ belongs to the space $\mathrm{VMO}(\boz)$
(see, for instance, \cite{s75}). In this article, our aim is to establish
the weighted Calder\'on--Zygmund type estimates
\begin{equation}\label{1.7}
\|\nabla u\|_{L^p_\omega(\boz;\rn)}\le C\|\mathbf{f}\|_{L^p_\omega(\boz;\rn)}
\end{equation}
for the Dirichlet problem \eqref{1.4}, with $p\in(1,\fz)$ and $\omega\in A_q(\rn)$
for some $q\in[1,\fz)$, and then give their applications, where $C$ is a positive constant
independent of $u$ and $\mathbf{f}$.

Let $n\ge2$, $\boz\subset\rn$ be a bounded domain, and the matrix $A:=a+b$ satisfy Assumption \ref{a1}.
For the Dirichlet problem $(D)_p$, the estimate \eqref{1.7} with $p\in(1,\fz)$ and $\omega\equiv1$
was established by Di Fazio \cite{d96}, under the assumptions that $a\in\mathrm{VMO}(\rn)$, $b\equiv0$,
and $\partial\boz\in C^{1,1}$, which was weakened to $\partial\boz\in C^{1}$ by Auscher and Qafsaoui \cite{aq02}.
Moreover, for the Dirichlet problem $(D)_p$,  the estimate \eqref{1.7}
with $p\in(1,\fz)$ and $\omega\equiv1$ was obtained by Byun and Wang in \cite{b05,bw04},
under the assumptions that $a$ satisfies the $(\dz,R)$-BMO
condition for sufficiently small $\dz\in(0,\fz)$, $b\equiv0$, and $\boz$ is a bounded Lipschitz domain
with a small Lipschitz constant or a bounded Reifenberg flat domain (see, for instance, \cite{r60,t97}).
Furthermore, for the Dirichlet problem $(D)_p$ with $a$ having partial small $\mathrm{BMO}$ coefficients
and $b\equiv0$, the estimate \eqref{1.7} with $p\in(1,\fz)$ and $\omega\equiv1$
was studied, respectively, by Dong and Kim \cite{dk10}, and Krylov \cite{k07},
under the assumption that $\boz$ is a bounded Lipschitz domain with small Lipschitz constant.
Moreover, if $\boz$ is a bounded quasi-convex domain, $a$ satisfies the $(\dz,R)$-BMO
condition for sufficiently small $\dz\in(0,\fz)$, and $b\equiv0$,
the estimate \eqref{1.7} with $p\in(1,\fz)$ and $\omega\equiv1$ was
established by Jia et al. \cite{jlw10} for the Dirichlet problem $(D)_p$.
For the Dirichlet problem $(D)_p$ in a general Lipschitz domain $\boz$,
it was proved by Shen \cite{sh05a} that, if $a\in\mathrm{VMO}(\rn)$ and $b\equiv0$, then
\eqref{1.7} with $\omega\equiv1$ holds true for any given $p\in(\frac32-\uc,3+\uc)$ when $n\ge3$, or
$p\in(\frac43-\uc,4+\uc)$ when $n=2$, where $\uc\in(0,\fz)$ is a positive constant
depending on $\boz$. It is worth pointing out that, when $A:=I$ (the identity matrix)
in \eqref{1.7}, the range of $p$ obtained in \cite{sh05a} is even sharp for
general Lipschitz domains (see, for instance, \cite{jk95}).

For the weighted Dirichlet problem $(D)_{p,\,\omega}$ with $a$ having partial small
$\mathrm{BMO}$ coefficients and $b\equiv0$, \eqref{1.7} with $p\in(2,\fz)$
and $\omega\in A_{p/2}(\rn)$ was obtained by Byun and Palagachev \cite{bp14} under the assumption that
$\boz$ is a bounded Reifenberg flat domain. Furthermore,
for the problem $(D)_{p,\,\omega}$ with $a$ having partial small
$\mathrm{BMO}$ coefficients and $b\equiv0$, the estimate \eqref{1.7} with $p\in(1,\fz)$
and $\omega\in A_{p}(\rn)$ was established by Dong and Kim \cite{dk18} under the assumption that
$\boz$ is a bounded Reifenberg flat domain. For the problem $(D)_{p,\,\omega}$ with
$a$ having small $\mathrm{BMO}$ coefficients and $b\equiv0$, \eqref{1.7} with $p\in(1,\fz)$ and
$\omega\in A_{p}(\rn)$ was obtained by Adimurthi et al. \cite{amp18}
under the assumption that $\boz$ is a bounded Lipschitz domain with small Lipschitz constant.

Now, we state the main results of this article as follows.

\begin{theorem}\label{t1.1}
Let $n\ge2$, $\boz\subset\rn$ be a bounded $\mathrm{NTA}$ domain, the matrix $A$ satisfy Assumption \ref{a1},
and $p\in(2,\fz)$. Assume that there exist positive constants $C_0\in(0,\fz)$ and $r_0\in(0,\diam(\boz))$
such that, for any ball $B(x_0,r)$ having the property that $r\in(0,r_0/4)$
and either $x_0\in\partial\boz$ or $B(x_0,2r)\subset\boz$,
the weak reverse H\"older inequality
\begin{align}\label{1.8}
\lf[\frac{1}{|B_\boz(x_0,r)|}\int_{B_\boz(x_0,r)}|\nabla v(x)|^p\,dx\r]^{\frac1p}\le
C_0\lf[\frac{1}{|B_\boz(x_0,2r)|}\int_{B_\boz(x_0,2r)}
|\nabla v(x)|^{2}\,dx\r]^{\frac1{2}}
\end{align}
holds true for any function $v\in W^{1,2}(B_\boz(x_0,2r))$ satisfying
$\mathrm{div}(A\nabla v)=0$ in $B_\boz(x_0,2r)$, and
$v=0$ on $B(x_0,2r)\cap\partial\boz$ when $x_0\in\partial\boz$, where $B_\boz(x_0,r):=B(x_0,r)\cap\boz$.

\begin{itemize}
\item[\rm(i)] Then the weak solution $u\in W^{1,2}_0(\boz)$ of the Dirichlet problem
$(D)_p$ with $\mathbf{f}\in L^p(\boz;\rn)$ exists and, moreover, $u\in W^{1,p}_0(\boz)$ and
there exists a positive constant $C$, depending only on $n$, $p$, and $\boz$, such that
\begin{equation}\label{1.9}
\|\nabla u\|_{L^p(\boz;\rn)}\le C\|\mathbf{f}\|_{L^p(\boz;\rn)}.
\end{equation}

\item[\rm(ii)] Let $q\in[2,p]$, $q_0\in[1,\frac{q}{p'}]$, $r_0\in[(\frac{p}{q})',\fz]$,
and $\omega\in A_{q_0}(\rn)\cap RH_{r_0}(\rn)$.
Then a weak solution $u$ of the weighted Dirichlet problem
$(D)_{q,\,\omega}$ with $\mathbf{f}\in L^q_\omega(\boz;\rn)$
exists and, moreover, $u\in W^{1,q}_{0,\,\omega}(\boz)$ and there exists a positive constant $C$,
depending only on $n$, $p$, $q$, $[\omega]_{A_{q_0}(\rn)}$,
$[\omega]_{RH_{r_0}(\rn)}$, and $\boz$, such that
\begin{equation}\label{1.10}
\|\nabla u\|_{L^q_\omega(\boz;\rn)}\le C\|\mathbf{f}\|_{L^q_\omega(\boz;\rn)}.
\end{equation}
Here and thereafter, for any $s\in[1,\fz]$, $s'$ denotes its conjugate exponent, namely, $1/s + 1/s'= 1$.
\end{itemize}
\end{theorem}

We prove Theorem \ref{t1.1} via using a (weighted) real-variable argument
(see Theorem \ref{t3.1} below), which was essentially established
in \cite[Theorem 3.4]{sh07} (see also \cite{g12,g18,sh18,sh05a,ycyy20}) and inspired by \cite{cp98,w03}.
It is worth pointing out that a similar real-variable argument with the different motivation was
used in \cite{a07,am07}. Moreover, a different weighted real-variable argument
was obtained by Shen \cite[Theorem 2.1]{sh20}. Furthermore, the linear structure
of second order elliptic operators of divergence form and the properties of Muckenhoupt weights
are subtly used in the proof of Theorem \ref{t1.1}.

Let the $(\dz,R)$-$\mathrm{BMO}$ condition and the space $\mathrm{VMO}(\boz)$ be as in Definition \ref{d2.1} below.
As an application of Theorem \ref{t1.1}, we obtain the (weighted) global gradient estimates for solutions
to Dirichlet boundary problems on bounded Lipschitz domains as follows.

\begin{theorem}\label{t1.2}
Let $n\ge2$, $\boz\subset\rn$ be a bounded Lipschitz domain, and the matrix $A$ satisfy Assumption \ref{a1}.
\begin{itemize}
\item[\rm(i)] Then there exist positive constants $\uc_0,\ \dz_0\in(0,\fz)$,
depending only on $n$ and the Lipschitz constant of $\boz$, such that, for any given
$p\in((3+\uc_0)', 3+\uc_0)$ when $n\ge3$, or $p\in((4+\uc_0)', 4+\uc_0)$ when $n=2$,
if $A$ satisfies the $(\dz,R)$-$\mathrm{BMO}$
condition for some $\dz\in(0,\dz_0)$ and $R\in(0,\fz)$, or $A\in\mathrm{VMO}(\boz)$, then
the Dirichlet problem $(D)_{p}$ with $\mathbf{f}\in L^p(\boz;\rn)$ is uniquely solvable and
there exists a positive constant $C$, depending only on $n$, $p$, and the Lipschitz constant of $\boz$,
such that, for any weak solution $u$, $u\in W^{1,p}_0(\boz)$ and
\begin{equation}\label{1.11}
\|\nabla u\|_{L^p(\boz;\rn)}\le C\|\mathbf{f}\|_{L^p(\boz;\rn)}.
\end{equation}

\item[\rm(ii)] Let $\uc_0$ be as in (i) and $p_0:=3+\uc_0$ when $n\ge3$, and $p_0:=4+\uc_0$ when $n=2$.
Then, for any given $p\in(p_0',p_0)$ and any $\omega\in A_{\frac{p}{p_0'}}(\rn)\cap RH_{(\frac{p_0}{p})'}(\rn)$,
there exists a positive constant $\dz_0\in(0,\fz)$,
depending only on $n$, $p$, the Lipschitz constant of $\boz$, $[\omega]_{A_{\frac{p}{p_0'}}(\rn)}$,
and $[\omega]_{RH_{(\frac{p_0}{p})'}(\rn)}$,  such that, if
$A$ satisfies the $(\dz,R)$-$\mathrm{BMO}$ condition for some $\dz\in(0,\dz_0)$
and $R\in(0,\fz)$, or $A\in\mathrm{VMO}(\boz)$,
then the weighted Dirichlet problem $(D)_{p,\,\omega}$ with
$\mathbf{f}\in L^p_\omega(\boz;\rn)$ is uniquely solvable and there exists a positive constant
$C$, depending only on $n$, $p$,
$[\omega]_{A_{\frac{p}{p_0'}}(\rn)}$, $[\omega]_{RH_{(\frac{p_0}{p})'}(\rn)}$,
and the Lipschitz constant of $\boz$, such that,
for any weak solution $u$, $u\in W^{1,p}_{0,\,\omega}(\boz)$ and
\begin{equation}\label{1.12}
\|\nabla u\|_{L^p_\omega(\boz;\rn)}\le C\|\mathbf{f}\|_{L^p_\omega(\boz;\rn)}.
\end{equation}
\end{itemize}
\end{theorem}

Let $A:=a+b$ satisfy Assumption \ref{a1} and $p_0$ be as in Theorem \ref{t1.2}(ii).
The key to proving Theorem \ref{t1.2} is to show that
the weak reverse H\"older inequality \eqref{1.8} is valid for any
$p\in(2,p_0)$. To do this, we flexibly apply the real-variable argument established in Theorem \ref{t3.1} below,
the method of perturbation, the assumptions that $b\in\mathrm{BMO}(\boz;\rr^{n^2})$ and $b$ is anti-symmetric,
and the properties of Muckenhoupt weights.

\begin{remark}\label{r1.3}
Let $n\ge2$, $\boz\subset\rn$ be a bounded Lipschitz domain, and $A:=a+b$ satisfy Assumption \ref{a1}.
If $a\in\mathrm{VMO}(\boz)$ and $b\equiv0$, then
Theorem \ref{t1.2} in this case was essentially established by Shen \cite[Theorem C]{sh05a}.
Thus, Theorem \ref{t1.2} improves \cite[Theorem C]{sh05a} via weakening the condition
for the matrix $A$.
\end{remark}

Let the quasi-convex domain be as in Definition \ref{d2.2} below.
We further obtain the following (weighted) global gradient estimates for the
Dirichlet problems on quasi-convex domains by using Theorem \ref{t1.1}.

\begin{theorem}\label{t1.3}
Let $n\ge2$, $\boz\subset\rn$ be a bounded $\mathrm{NTA}$ domain, and $p\in(1,\fz)$.
Assume further that the matrix $A$ satisfies Assumption \ref{a1}, and
$\boz$ is a $(\dz,\sz,R)$ quasi-convex domain with some $\dz,\ \sz\in(0,1)$ and $R\in(0,\fz)$.
\begin{itemize}
\item[\rm(i)] Then there exists a positive constant $\dz_0\in(0,1)$,
depending only on $n$, $p$, and $\boz$, such that, if $\boz$ is a $(\dz,\sz,R)$ quasi-convex domain
and $A$ satisfies the $(\dz,R)$-$\mathrm{BMO}$
condition for some $\dz\in(0,\dz_0)$, $\sz\in(0,1)$, and $R\in(0,\fz)$, or $A\in\mathrm{VMO}(\boz)$, then
the Dirichlet problem $(D)_{p}$ with $\mathbf{f}\in L^p(\boz;\rn)$ is uniquely solvable and,
for any weak solution $u$ of the problem $(D)_{p}$, $u\in W^{1,p}_0(\boz)$ and
\begin{equation*}
\|\nabla u\|_{L^{p}(\boz;\rn)}\le C\|\mathbf{f}\|_{L^p(\boz;\rn)},
\end{equation*}
where $C$ is a positive constant depending only on $n$, $p$, and $\boz$.

\item[\rm(ii)] Let $\omega\in A_p(\rn)$. Then there exists a positive constant $\dz_0\in(0,1)$,
depending only on $n$, $p$, $\boz$, and $[\omega]_{A_p(\rn)}$,  such that, if
$\boz$ is a $(\dz,\sz,R)$ quasi-convex domain and $A$ satisfies the $(\dz,R)$-$\mathrm{BMO}$
condition for some $\dz\in(0,\dz_0)$, $\sz\in(0,1)$,
and $R\in(0,\fz)$, or $A\in\mathrm{VMO}(\boz)$,  then the weighted Dirichlet
problem $(D)_{p,\,\omega}$ with
$\mathbf{f}\in L^p_\omega(\boz;\rn)$ is uniquely solvable and there exists a positive constant
$C$, depending only on $n$, $p$, $[\omega]_{A_{p}(\rn)}$, and $\boz$, such that,
for any weak solution $u$, $u\in W^{1,p}_{0,\,\omega}(\boz)$ and
\begin{equation}\label{1.13}
\|\nabla u\|_{L^p_\omega(\boz;\rn)}\le C\|\mathbf{f}\|_{L^p_\omega(\boz;\rn)}.
\end{equation}
\end{itemize}
\end{theorem}

Assume that $p\in[2,\fz)$. To prove Theorem \ref{t1.3} via using Theorem \ref{t1.1}, we need to prove that
there exists a $\dz_0\in(0,1)$, depending on $p$ and $\boz$, such that,
if $\boz$ is a bounded $(\dz,\,\sz,\,R)$ quasi-convex domain with some $\dz\in(0,\dz_0)$,
$\sz\in(0,1)$, and $R\in(0,\fz)$, then the weak reverse H\"older inequality \eqref{1.8} is valid for
the exponent $p$. To this end, we adequately use the real-variable argument obtained in Theorem \ref{t3.1} below,
the method of perturbation, and the geometric properties of quasi-convex domains
(see Lemma \ref{l5.3} below).

Let the (semi-)convex domain be as in Remark \ref{r2.3}(i) below.
As a corollary of Theorem \ref{t1.3}, we have the following conclusion.

\begin{corollary}\label{c1.1}
Let $n\ge2$, $\boz\subset\rn$ be a bounded $C^1$ or (semi-)convex domain,
$p\in(1,\fz)$, and $\omega\in A_p(\rn)$. Assume that the matrix $A$ satisfies Assumption \ref{a1}.
Then there exists a positive constant $\dz_0\in(0,\fz)$,
depending only on $n$, $p$, $\boz$, and $[\omega]_{A_p(\rn)}$, such that, if
$A$ satisfies the $(\dz,R)$-$\mathrm{BMO}$ condition for some $\dz\in(0,\dz_0)$
and $R\in(0,\fz)$, or $A\in\mathrm{VMO}(\boz)$, then the weighted Dirichlet problem $(D)_{p,\,\omega}$ with
$\mathbf{f}\in L^p_\omega(\boz;\rn)$ is uniquely solvable and there exists a positive constant
$C$, depending only on $n$, $p$, $[\omega]_{A_p(\rn)}$, and $\boz$, such that,
for any weak solution $u$, $u\in W^{1,p}_{0,\,\omega}(\boz)$ and
$$\|\nabla u\|_{L^p_\omega(\boz;\rn)}\le C\|\mathbf{f}\|_{L^p_\omega(\boz;\rn)}.$$
\end{corollary}

Moreover, let the Reifenberg flat domain be as in Remark \ref{r2.3}(ii) below.
By Theorem \ref{t1.3} and Remark \ref{r2.3}(ii), we have the following corollary.

\begin{corollary}\label{c1.2}
Let $n\ge2$, $\boz\subset\rn$ be a bounded $(\dz,R)$-Reifenberg flat domain
with some $\dz\in(0,1)$ and $R\in(0,\fz)$, $p\in(1,\fz)$,
and $\omega\in A_p(\rn)$. Assume that the matrix $A$ satisfies Assumption \ref{a1}.
Then there exists a positive constant $\dz_0\in(0,1)$,
depending only on $n$, $p$, $\boz$, and $[\omega]_{A_p(\rn)}$, such that, if
$\boz$ is a bounded $(\dz,R)$-Reifenberg flat domain and $A$ satisfies the $(\dz,R)$-$\mathrm{BMO}$
condition for some $\dz\in(0,\dz_0)$ and $R\in(0,\fz)$, or $A\in\mathrm{VMO}(\boz)$,
then the weighted Dirichlet problem $(D)_{p,\,\omega}$ with
$\mathbf{f}\in L^p_\omega(\boz;\rn)$ is uniquely solvable and there exists a positive constant
$C$, depending only on $n$, $p$, $[\omega]_{A_{p}(\rn)}$, and $\boz$, such that,
for any weak solution $u$, $u\in W^{1,p}_{0,\,\omega}(\boz)$ and
\begin{equation*}
\|\nabla u\|_{L^p_\omega(\boz;\rn)}\le C\|\mathbf{f}\|_{L^p_\omega(\boz;\rn)}.
\end{equation*}
\end{corollary}

\begin{remark}\label{r1.4}
Let $n\ge2$ and $\boz\subset\rn$ be a bounded $(\dz,\sz,R)$ quasi-convex domain
with some $\dz,\,\sz\in(0,1)$ and $R\in(0,\fz)$.
Assume that the matrix $A:=a+b$ satisfies Assumption \ref{a1}.
\begin{itemize}
\item[\rm(i)] If $a$ satisfies the $(\dz,R)$-$\mathrm{BMO}$
condition for some sufficiently small $\dz\in(0,\fz)$ and $b\equiv0$,
then Theorem \ref{t1.3}(i) in this case was established by Jia et al. in \cite[Theorem 1.1]{jlw10}.
Thus, Theorem \ref{t1.3}(i) improves \cite[Theorem 1.1]{jlw10} via weakening the condition for the matrix $A$.
Moreover, even when $b\equiv0$, the conclusion of Theorem \ref{t1.3}(ii) in this case
is also new.

\item[\rm(ii)] When $a$ satisfies the $(\dz,R)$-$\mathrm{BMO}$
condition for some sufficiently small $\dz\in(0,\fz)$, $b\equiv0$, and $\omega\equiv1$,
Corollary \ref{c1.2} in this case was obtained by Byun and Wang in \cite[Theorem 1.5]{bw04}.
Therefore, even when $\omega\equiv1$, Corollary \ref{c1.2} in this case also improves \cite[Theorem 1.5]{bw04}
via weakening the assumption on $A$.

Furthermore, we point out that the approach used in this article to establish the
global gradient estimates is different from that used in \cite{b05,bw04,jlw10}.
Indeed, the global estimates were obtained in \cite{b05,bw04,jlw10} by using an approximation
argument, the modified Vitali covering lemma, and a compactness method.
However, in this article, we establish the (weighted) global estimates via using
a (weighted) real-variable argument (see Theorem \ref{t3.1} below),
the method of perturbation, and the geometric properties of quasi-convex domains.

\item[\rm(iii)] We point out that, even when $\boz$ is a bounded convex domain and $\omega\equiv1$,
the conclusion of Corollary \ref{c1.1} in this case is also new.
\end{itemize}
\end{remark}

Applying the weighted global regularity estimates obtained in Theorems \ref{t1.2}(ii)
and \ref{t1.3}(ii), and some tools from harmonic analysis, such as the properties
of Muckenhoupt weights, the interpolation theorem of operators,
and the Rubio de Francia extrapolation theorem,
we obtain the global gradient estimates for the Dirichlet problem \eqref{1.4},
respectively, in (weighted) Lorentz spaces, (Lorentz--)Morrey spaces,
(Musielak--)Orlicz spaces, and variable Lebesgue spaces, which have independent interests
and are presented in Section \ref{s6} below.
It is worth pointing out that the approach used in this
article to establish the global estimates in both Orlicz spaces
and variable Lebesgue spaces is quite different from that used in \cite{bow14,byz08}.
In \cite{bow14,byz08}, the global estimates in variable Lebesgue spaces or in Orlicz spaces
were established via the so-called ``maximum function free technique".
However, in this article, we obtain the global gradient estimates in both Orlicz spaces and
variable Lebesgue spaces by simply using weighted global estimates in
Theorems \ref{t1.2}(ii) and \ref{t1.3}(ii), and the Rubio de Francia extrapolation theorem.
Furthermore, we point out that the extrapolation theorem used in this article
is also valid for the boundary value problem studied in
\cite{bow14,byz08} and independent of the boundary value condition
and the considered equation.

Moreover, the global estimates in Orlicz spaces were obtained in \cite{jlw07} through an approximation
argument, the modified Vitali covering lemma, and a compactness method.
However, in this article, the global gradient estimates in Orlicz spaces
are obtained as corollaries of both the weighted norm inequality in
Theorem \ref{t1.3}(ii) and the extrapolation theorem.

The remainder of this article is organized as follows.

In Section \ref{s2}, we present several notions on the $(\dz,R)$-$\mathrm{BMO}$
condition, the space $\mathrm{VMO}(\boz)$, and several domains, and clarify their
relations. We then prove Theorems \ref{t1.1}, and \ref{t1.2}, and \ref{t1.3}, and
Corollary \ref{c1.1}, respectively, in Sections \ref{s3}, and \ref{s4}, and \ref{s5}.
Several applications of Theorems \ref{t1.2} and \ref{t1.3} are given in Section \ref{s6}.

Finally, we make some conventions on notation.
Throughout this article, we always denote by $C$ a \emph{positive constant} which is
independent of the main parameters, but it may vary from line to
line. We also use $C_{(\gz,\,\bz,\,\ldots)}$ or $c_{(\gz,\,\bz,\,\ldots)}$
to denote a  \emph{positive constant} depending on the indicated parameters $\gz,$ $\bz$,
$\ldots$. The \emph{symbol} $f\ls g$ means that $f\le Cg$. If $f\ls
g$ and $g\ls f$, then we write $f\sim g$. If $f\le Cg$ and $g=h$ or $g\le h$, we then write $f\ls g\sim h$
or $f\ls g\ls h$, \emph{rather than} $f\ls g=h$ or $f\ls g\le h$.
For each ball $B:=B(x_B,r_B)$ of $\rn$, with some $x_B\in\rn$ and
$r_B\in (0,\fz)$, and $\az\in(0,\fz)$, let $\az B:=B(x_B,\az r_B)$;
furthermore, denote the set $B(x,r)\cap\boz$ by $B_\boz(x,r)$, and the set $(\az B)\cap\boz$ by $\az B_\boz$.
For any subset $E$ of $\rn$, we denote the \emph{set} $\rn\setminus E$ by $E^\complement$,
and its \emph{characteristic function} by $\mathbf{1}_{E}$ .
For any $\omega\in A_p(\rn)$ with $p\in[1,\fz)$, and any measurable set $E\subset\rn$,
let $\omega(E):=\int_E\omega(x)\,dx$. For any given $q\in[1,\fz]$, we denote by $q'$
its \emph{conjugate exponent}, namely, $1/q + 1/q'= 1$.
Finally, for any measurable set $E\subset\rn$, $\omega\in A_q(\rn)$ with some $q\in[1,\fz)$,
and $f\in L^1(E)$, we denote the integral $\int_E|f(x)|\omega(x)\,dx$
simply by $\int_E|f|\omega\,dx$ and, when $|E|<\fz$, we use the notation
$$\fint_E fdx:=\frac{1}{|E|}\int_Ef(x)dx.$$

\section{Several notions\label{s2}}

In this section, we present the definitions of the $(\dz,R)$-$\mathrm{BMO}$ condition,
the space $\mathrm{VMO}(\boz)$, and several domains including quasi-convex domains,
(semi-)convex domains, and Reifenberg flat domains. Furthermore, we also clarify the relations between NTA domains,
Lipschitz domains, quasi-convex domains, Reifenberg flat domains, and semi-convex domains.

First, we recall the notions of the $(\dz,R)$-$\mathrm{BMO}$ condition and
the space $\mathrm{VMO}(\boz)$ as follows (see, for instance, \cite{bw05,bw04,s75}).

\begin{definition}\label{d2.1}
Let $\boz\subset\rn$ be a domain and $R,\,\dz\in(0,\fz)$.
\begin{itemize}
\item[{\rm(i)}] A function $f\in L^1_{\loc}(\boz)$ is said to satisfy
the \emph{$(\dz,R)$-$\mathrm{BMO}$ condition} if
\begin{equation*}
\|f\|_{\ast,\,R}:=\sup_{B(x,r)\subset\boz,\,r\in(0,R)}\frac{1}{|B(x,r)|}
\int_{B(x,r)}|f(y)-f_{B(x,r)}|\,dy\le\dz,
\end{equation*}
where the supremum is taken over all balls $B(x,r)\subset\boz$ with $r\in(0,R)$.
Furthermore, $f$ is said to belong to the \emph{space} $\mathrm{VMO}(\boz)$ if $f$ satisfies the
$(\dz,R)$-$\mathrm{BMO}$ condition for some $\dz,\,R\in(0,\fz)$, and
$$\lim_{r\to0^{+}}\sup_{B(x,r)\subset\boz}\frac{1}{|B(x,r)|}\int_{B(x,r)}|f(y)-f_{B(x,r)}|\,dy=0,
$$
where $r\to0^{+}$ means $r\in(0,\fz)$ and $r\to0$.

\item[{\rm(ii)}] A matrix $A:=\{a_{ij}\}_{i,j=1}^n$ is said to
satisfy the \emph{$(\dz,R)$-$\mathrm{BMO}$ condition} [resp., $A\in\mathrm{VMO}(\boz)$]
if, for any $i,\,j\in\{1,\,\ldots,\,n\}$, $a_{ij}$
satisfies the $(\dz,R)$-$\mathrm{BMO}$ condition [resp., $a_{ij}\in\mathrm{VMO}(\boz)$].
\end{itemize}
\end{definition}

\begin{remark}\label{r2.1}
Let $\boz\subset\rn$ be a domain and $\dz\in(0,\fz)$. If $f\in\mathrm{BMO}(\boz)$ and $\|f\|_{\mathrm{BMO}(\boz)}\le\dz$,
then $f$ satisfies the $(\dz,R)$-$\mathrm{BMO}$ condition for any $R\in(0,\fz)$.
Moreover, if $f\in\mathrm{VMO}(\boz)$, then $f$ satisfies the $(\gz,R)$-$\mathrm{BMO}$
condition for any $\gz\in(0,\fz)$ and some $R\in(0,\fz)$.
\end{remark}

Now, we recall the notion of quasi-convex domains as follows.
Let $E_1,\,E_2\subset\rn$ be non-empty measurable subsets. Then the \emph{Hausdorff distance} between
$E_1$ and $E_2$ is defined by setting
$$d_H(E_1,E_2):=\max\lf\{\sup_{x\in E_1}\inf_{y\in E_2}|x-y|,\ \sup_{y\in E_2}\inf_{x\in E_1}|x-y|\r\}.
$$

\begin{definition}\label{d2.2}
Let $\boz\subset\rn$ be a domain, $\dz,\ \sz\in(0,1)$, and $R\in(0,\fz)$.
Then $\boz$ is called a \emph{$(\dz,\,\sz,\,R)$ quasi-convex domain} if,
for any $x\in\partial\boz$ and $r\in(0,R]$,
\begin{itemize}
\item[\rm(a)] there exists an $x_0\in\boz$, depending on $x$,
such that $B(x_0,\sz r)\subset\boz\cap B(x,r)$;
\item[\rm(b)] there exists a convex domain $V:=V(x,r)$, depending on $x$ and $r$,
such that $B(x,r)\cap\boz\subset V$ and $d_H(\partial(B(x,r)\cap\boz),\partial V)\le\dz r$.
\end{itemize}
\end{definition}

\begin{remark}\label{r2.2}
\begin{itemize}
\item[\rm(i)] The notion of quasi-convex domains was introduced by Jia et al. \cite{jlw10}
to study the global regularity of second order elliptic equations.
Roughly speaking, a quasi-convex domain is a domain satisfying that the local boundary
is close to be convex at small scales.
It is easy to see that, if $\boz$ is a convex domain, then $\boz$ is a $(\dz,\,\sz,\,R)$ quasi-convex domain
for any $\dz\in(0,1)$, some $\sz\in(0,1)$, and some $R\in(0,\fz)$.
\item[\rm(ii)] We may \emph{always} assume that the convex domain $V$ in
Definition \ref{d2.2}(b) is the convex hull of $B(x,r)\cap\boz$,
which is the smallest convex domain containing $B(x,r)\cap\boz$
(see \cite[Theorem 3.1]{jlw10} and \cite[Remark 1.2]{z19} for more details).
\item[\rm(iii)] It was showed in \cite[Theorem 3.10]{jlw10}
that Definition \ref{d2.2}(b) can be replaced by the following condition:
\begin{itemize}
\item[\rm(c)] For any $x\in\partial\boz$ and $r\in(0,R]$,
there exist an $(n-1)$-dimensional plane $L(x,r)$ containing $x$,
a choice of the unit normal vector to $L(x,r)$, denoted by $\boldsymbol{\nu}_{x,r}$,
and the half space
$$H(x,r):=\{y+t\boldsymbol{\nu}_{x,t}:\ y\in L(x,t),\ t\in[-\dz r,\fz)\}
$$
such that
$$\boz\cap B(x,r)\subset H(x,r)\cap B(x,r).
$$
\end{itemize}

More precisely, it was proved in \cite[Theorem 3.10]{jlw10} that, if
$\boz$ is a domain satisfying the assumptions (a) and (c) with some $\dz,\ \sz\in(0,1)$
and $R\in(0,\fz)$, then $\boz$ is a $(\dz_1,\,\sz,\,R)$ quasi-convex domain
with $\dz_1:=8\dz/\sz$.
\end{itemize}
\end{remark}

In the following remark, we recall the notions of
semi-convex domains and Reifenberg flat domains,
and then clarify the relations between NTA domains,
Lipschitz domains, quasi-convex domains, Reifenberg
flat domains, and semi-convex domains.

\begin{remark}\label{r2.3}
\begin{enumerate}
\item[(i)] A set $E\subset\rn$ is said to satisfy an \emph{exterior ball condition} at $x\in\partial E$
if there exist a $\mathbf{v}\in S^{n-1}$ and an $r\in(0,\fz)$ such that
\begin{equation}\label{2.1}
B(x+r\mathbf{v},r)\subset(\rn\setminus E),
\end{equation}
where $S^{n-1}$ denotes the  \emph{unit sphere} of $\rn$.
For such an $x\in \partial E$, let
$$r(x):=\sup\lf\{r\in(0,\fz):\ \eqref{2.1} \ \text{holds true for some}\ v\in S^{n-1}\r\}.
$$
A set $E$ is said to satisfy a \emph{uniform exterior ball condition} (for short,
UEBC) with radius $r\in(0,\fz]$ if
\begin{equation}\label{2.2}
\inf_{x\in\partial E}r(x)\ge r,
\end{equation}
and the value $r$ in \eqref{2.2} is referred to the UEBC \emph{constant}.
A set $E$ is said to satisfy a UEBC if there exists an $r\in(0,\fz]$ such that $E$ satisfies
the uniform exterior ball condition with radius $r$. Moreover, the largest positive constant $r$
as above is called the \emph{uniform ball constant} of $E$.

It is known that, for any open set $\boz\subset\rn$ with compact boundary,
$\boz$ is a \emph{Lipschitz domain satisfying a} UEBC if and only if $\boz$
is a \emph{semi-convex domain} in $\rn$ (see, for instance, \cite[Theorem 2.5]{mmmy10} or
\cite[Theorem 3.9]{mmy10}). Moreover, more equivalent characterizations of
semi-convex domains are given in \cite{mmmy10,mmy10}.

It is worth pointing out that, if $\boz\subset\rn$ is convex, then $\boz$ satisfies
a UEBC with the uniform ball constant $\fz$ (see, for instance, \cite{e58}).
Thus, convex domains of $\rn$ are semi-convex domains
(see, for instance, \cite{mmmy10,mmy10,y16,yyy18}). Moreover,
(semi-)convex domains are special cases of Lipschitz domains.
More precisely,
$$\text{class of convex domains}\subsetneqq\text{class of semi-convex domains}
\subsetneqq\text{class of Lipschitz domains}.
$$

\item[\rm(ii)] Let $n\ge2$, $\dz\in(0,1)$, and $R\in(0,\fz)$.
A domain $\boz\subset\rn$ is called a $(\dz,R)$-\emph{Reifenberg
flat domain} if, for any $x_0\in\partial\boz$ and $r\in(0,R]$, there exists
a system of coordinates, $\{y_1,\,\ldots,\,y_n\}$, which may depend on $x_0$ and $r$,
such that, in this coordinate system, $x_0=\mathbf{0}$ and
\begin{equation*}
B(\mathbf{0},r)\cap\{y\in\rn:\ y_n>\dz r\}\subset B(\mathbf{0},r)\cap\boz\subset
B(\mathbf{0},r)\cap\{y\in\rn:\ y_n>-\dz r\},
\end{equation*}
where $\mathbf{0}$ denotes the \emph{origin} of $\rr^{n}$.
The Reifenberg flat domain was introduced by Reifenberg in \cite{r60},
which naturally appears in the theory of minimal surfaces and free boundary problems.
A typical example of Reifenberg flat domains is the well-known Van Koch snowflake
(see, for instance, \cite{t97}).
In recent years, boundary value problems of elliptic or parabolic equations
on Reifenberg flat domains have been widely concerned and studied
(see, for instance, \cite{b18,bd17a,bd17,bow14,bw04,byz08,dk18,mp12,z16,z16a}).

Moreover, for any given $\dz\in(0,1)$ and $R\in(0,\fz)$,
a $(\dz,R)$-Reifenberg flat domain is a $(\dz,\sz,R)$ quasi-convex
domain with  $\sz:=\frac{1-\dz}{2}$ (see, for instance, \cite{jlw10}).
However, a quasi-convex domain may not be a Reifenberg flat domain.
Indeed, let
$$\boz:=\lf\{(x_1,x_2)\in\rr^2:\ 1>x_2>|x_1|\r\}.$$
Then $\boz$ is convex and hence a quasi-convex domain,
but $\boz$ is not a Reifenberg flat domain.
Thus,
$$\text{class of Reifenberg flat domains}\ \subsetneqq\ \text{class of quasi-convex domains}.$$
Furthermore, it was showed by Kenig and Toro \cite{kt97}
that, if $\dz$ is sufficiently small, then $(\dz,R)$-Reifenberg flat domains are also NTA domains.
\item[\rm(iii)] By the facts that Lipschitz domains with small Lipschitz constants are Reifenberg
flat domains and that $C^1$ domains are Lipschitz domains with small Lipschitz constants (see, for instance, \cite{t97}),
we conclude that $C^1$ domains are $(\dz,\sz,R)$ quasi-convex
domains with any $\dz\in(0,\dz_0)$, some $\sz\in(0,1)$, and some $R\in(0,\fz)$,
where $\dz_0\in(0,1)$ is a positive constant depending only on $\boz$.

Let $\boz\subset\rn$ be a bounded semi-convex domain. Then $\boz$ is a $(\dz,\sz,R)$ quasi-convex domain for any
$\dz\in(0,\dz_0)$, some $\sz\in(0,1)$, and some $R\in(0,\fz)$,
where $\dz_0\in(0,1)$ is a positive constant depending only on $\boz$
(see Lemma \ref{l5.4} below).

\item[\rm(iv)] On NTA domains, quasi-convex domains,
Reifenberg flat domains, Lipschitz domains, $C^1$ domains, and (semi-)convex domains,
we have the following relations.
\begin{enumerate}
\item[\rm(iv)$_1$]
\begin{align*}
\text{class of $C^1$ domains}\ &\subsetneqq\ \text{class of Lipschitz domains with small Lipschitz constants}\\
&\subsetneqq\ \text{class of Lipschitz domains}\\
&\subsetneqq\ \text{class of NTA domains};
\end{align*}
$$\text{class of (semi-)convex domains}\ \subsetneqq\ \text{class of Lipschitz domains}.$$

\item[\rm(iv)$_2$]
\begin{align*}
\text{class of $C^1$ domains}\ &\subsetneqq\ \text{class of Lipschitz domains with small Lipschitz constants}\\
 &\subsetneqq\ \text{class of Reifenberg flat domains}\\
 &\subsetneqq\ \text{class of quasi-convex domains};
\end{align*}
$$\text{class of (semi-)convex domains}\ \subsetneqq\ \text{class of quasi-convex domains}.$$

\item[\rm(iv)$_3$] Reifenberg flat domains or quasi-convex domains may not be Lipschitz domains,
and generally Lipschitz domains may also
not be Reifenberg flat domains (see, for instance, \cite{t97}).
Moreover, (semi-)convex domains may not be Lipschitz domains with
small Lipschitz constants, or Reifenberg flat domains.
\end{enumerate}
\end{enumerate}
\end{remark}

\section{Proof of Theorem \ref{t1.1}\label{s3}}

 In this section, we prove Theorem \ref{t1.1} via using a real-variable argument for
(weighted) $L^p(\boz)$ estimates,
which is inspired by the work of Caffarelli and Peral \cite{cp98} (see also \cite{w03}).
When $\boz$ is a bounded Lipschitz domain in $\rn$,
the conclusion of Theorem \ref{t3.1} was essentially established in \cite[Theorem 3.4]{sh07} (see also
\cite[Theorems 2.1 and 2.2]{g12}, \cite[Theorem 4.2.6]{sh18},
\cite[Theorem 3.3]{sh05a}, and \cite[Theorem 3.1]{ycyy20}). It is worth pointing out that
the proofs of \cite[Theorem 3.4]{sh07}, \cite[Theorem 4.2.6]{sh18}, and \cite[Theorem 3.1]{ycyy20}
are also valid in the case of bounded NTA domains. Thus, we omit the proof of Theorem \ref{t3.1} here.
Furthermore, we also mention that a similar argument with a different motivation was established in \cite{a07,am07}.
Moreover, a different weighted real-variable argument was obtained by Shen \cite[Theorem 2.1]{sh20}.

\begin{theorem}\label{t3.1}
Let $n\ge2$, $\boz\subset\rn$ be a bounded $\mathrm{NTA}$ domain,
$p_1,\,p_2\in[1,\fz)$ satisfy $p_2>p_1$,
$F\in L^{p_1}(\boz)$, and $f\in L^q(\boz)$ with some $q\in(p_1,p_2)$.
Assume that, for any ball $B:=B(x_B,r_B)\subset\rn$ having the property
that $|B|\le \bz_1|\boz|$ and either $2B\subset\boz$ or $x_B\in\partial\boz$,
there exist two measurable functions $F_B$ and $R_B$ on $2B$
such that $|F|\le|F_B|+|R_B|$ on $2B\cap\boz$,
\begin{align}\label{3.1}
\lf(\fint_{2B_\boz}|R_B|^{p_2}\,dx\r)^{\frac1{p_2}}
\le C_1\lf[\lf(\fint_{\bz_2 B_\boz}|F|^{p_1}\,dx\r)^{\frac1{p_1}}
+\sup_{\wz{B}\supset B}\lf(\fint_{\wz{B}_\boz}|f|^{p_1}\,dx\r)^{\frac1{p_1}}\r],
\end{align}
and
\begin{align}\label{3.2}
\lf(\fint_{2B_\boz}|F_B|^{p_1}\,dx\r)^{\frac1{p_1}}
\le \uc\lf(\fint_{\bz_2 B_\boz}|F|^{p_1}\,dx\r)^{\frac1{p_1}}
+C_2\sup_{\wz{B}\supset B}\lf(\fint_{\wz{B}_\boz}|f|^{p_1}\,dx\r)^{\frac1{p_1}},
\end{align}
where $C_1,\,C_2,\,\uc$, and $\bz_1<1<\bz_2$ are positive constants independent
of $F,\,f,\,R_B,\,F_B$, and $B$, and the suprema are taken over all balls $\wz{B}\supset B$.

Then, for any $\omega\in A_{q/p_1}(\rn)\cap RH_s(\rn)$
with $s\in((\frac{p_2}{q})',\fz]$, there exists a positive constant $\uc_0$,
depending only on $C_1,\,C_2,\,n,\,p_1,\,p_2,\,q,\,\bz_1,\,\bz_2$, $[\omega]_{A_{q/p_1}(\rn)}$, and
$[\omega]_{RH_s(\rn)}$, such that, if $\uc\in[0,\uc_0)$, then
\begin{align*}
\lf[\frac{1}{\omega(\boz)}\int_{\boz}|F|^q\omega\,dx\r]^{\frac1q}
\le C\lf\{\lf(\frac{1}{|\boz|}\int_{\boz}|F|^{p_1}\,dx\r)^{\frac1{p_1}}
+\lf[\frac{1}{\omega(\boz)}\int_{\boz}|f|^{q}
\omega\,dx\r]^{\frac{1}{q}}\r\},
\end{align*}
where $C$ is a positive constant depending only on $C_1$,
$C_2$, $n$, $p_1,\,p_2$, $q$, $\bz_1,\,\bz_2$,
$[\omega]_{A_{q/p_1}(\rn)}$, and $[\omega]_{RH_s(\rn)}$.
\end{theorem}

To show Theorem \ref{t1.1} via using Theorem \ref{t3.1},
we need the following auxiliary conclusion.

\begin{lemma}\label{l3.1}
Let $n\ge2$, $\boz\subset\rn$ be a bounded $\mathrm{NTA}$ domain,
$p\in(1,\fz)$, and $p'\in(1,\fz)$ be given by $1/p+1/p'=1$.
Assume that, if the matrix $A$ satisfies Assumption \ref{a1} and $\mathbf{f}\in L^p(\boz;\rn)$, then
the weak solution $u\in W^{1,p}_0(\boz)$ of the Dirichlet problem
\begin{equation}\label{3.3}
\begin{cases}
-\mathrm{div}(A\nabla u)=\mathrm{div}(\mathbf{f})\ \ & \text{in}\ \ \boz,\\
u=0\ \ & \text{on}\ \ \partial\boz
\end{cases}
\end{equation}
satisfies the estimate
\begin{equation}\label{3.4}
\|\nabla u\|_{L^p(\boz;\rn)}\le
C\|\mathbf{f}\|_{L^p(\boz;\rn)},
\end{equation}
where $C$ is a positive constant independent of $u$
and $\mathbf{f}$.
Let $\mathbf{g}\in L^{p'}(\boz;\rn)$ and
$v\in W^{1,p'}_0(\boz)$ be the weak solution of the Dirichlet problem
\eqref{3.3} with $\mathbf{f}$ replaced by $\mathbf{g}$.
Then
$$\|\nabla v\|_{L^{p'}(\boz;\rn)}\le C\|\mathbf{g}\|_{L^{p'}(\boz;\rn)},$$
where $C$ is a positive constant independent of $v$ and $\mathbf{g}$.
\end{lemma}

\begin{proof}
Let $\mathbf{f}\in L^p(\boz;\rn)$ and
$w\in W^{1,p}_0(\boz)$ be the weak solution of the Dirichlet problem
\begin{equation}\label{3.5}
\begin{cases}
-\mathrm{div}(A^\ast\nabla w)=\mathrm{div}(\mathbf{f})\ \ & \text{in}\ \ \boz,\\
w=0\ \ & \text{on}\ \ \partial\boz,
\end{cases}
\end{equation}
where $A^\ast$ denotes the transpose of $A$.
By the assumption that $A$ satisfies Assumption \ref{a1}, we find that
$A^\ast$ also satisfies Assumption \ref{a1}, which, combined with
the assumption \eqref{3.4}, further implies that
\eqref{3.4} also holds true for the weak solution $w$ of \eqref{3.5}.
Moreover, from the fact that $v$ is the weak solution of \eqref{3.3} with
$\mathbf{f}$ replaced by $\mathbf{g}$, it follows that
$$\int_\boz\mathbf{g}\cdot \nabla w\,dx=-\int_\boz A\nabla v\cdot\nabla w\,dx
=-\int_\boz A^\ast\nabla w\cdot\nabla v\,dx=
\int_\boz \mathbf{f}\cdot\nabla v\,dx,
$$
which, together with \eqref{3.4} and the H\"older inequality, further implies that
\begin{align*}
\|\nabla v\|_{L^{p'}(\boz;\rn)}&=\sup_{\|\mathbf{f}\|_{L^{p}(\boz;\rn)}\le1}
\lf|\int_\boz\mathbf{f}\cdot\nabla v\,dx\r|
=\sup_{\|\mathbf{f}\|_{L^{p}(\boz;\rn)}\le1}
\lf|\int_\boz\mathbf{g}\cdot \nabla w\,dx\r|\\
&\le\sup_{\|\mathbf{f}\|_{L^{p}(\boz;\rn)}\le1}\|\mathbf{g}\|_{L^{p'}(\boz;\rn)}
\|\nabla w\|_{L^{p}(\boz;\rn)}
\nonumber\\
&\ls\sup_{\|\mathbf{f}\|_{L^{p}(\boz;\rn)}\le1}\|\mathbf{g}\|_{L^{p'}(\boz;\rn)}
\|\mathbf{f}\|_{L^{p}(\boz;\rn)}\ls\|\mathbf{g}\|_{L^{p'}(\boz;\rn)}.\nonumber
\end{align*}
This finishes the proof of Lemma \ref{l3.1}.
\end{proof}

Moreover, to prove Theorem \ref{t1.1}, we need the following properties of $A_p(\rn)$ weights,
which are well known (see, for instance, \cite[Chapter 7]{g14} and \cite[Proposition 2.1 and Lemma 4.4]{am07}).

\begin{lemma}\label{l3.2}
Let $q\in(1,\fz)$, $\omega\in A_q(\rn)$, and $\boz\subset\rn$ be a bounded domain.
\begin{enumerate}
\item[\rm(i)] There exists a $q_1\in(1,q)$, depending only on $n$,
$q$, and $[\omega]_{A_q(\rn)}$, such that $\omega\in A_{q_1}(\rn)$.
\item[\rm(ii)] There exists a $\gamma\in(0,\fz)$,
depending only on $n$, $q$, and $[\omega]_{A_q(\rn)}$, such that $\omega\in RH_{1+\gamma}(\rn)$.
\item[\rm(iii)] If $q'$ denotes the conjugate number of $q$, namely, $1/q+1/q'=1$,
then $\omega^{-q'/q}\in A_{q'}(\rn)$ and
$[\omega^{-q'/q}]_{A_{q'}(\rn)}=[\omega]^{q'/q}_{A_q(\rn)}$.
\item[\rm(iv)] If $\omega\in RH_p(\rn)$ with $p\in(1,\fz)$, then there exists a $p_1\in(p,\fz]$,
depending only on $n$, $p$, and $[\omega]_{RH_p(\rn)}$, such that $\omega\in RH_{p_1}(\rn)$.
\item[\rm(v)] Let $p_0,\,q_0\in(1,\fz)$, $p\in(p_0,q_0)$, and $v\in L^1_\loc(\rn)$.
Then $v\in A_{\frac{p}{p_0}}(\rn)\cap RH_{(\frac{q_0}{p})'}(\rn)$ if and only if
$v^{1-p'}\in A_{\frac{p'}{(q_0)'}}(\rn)\cap RH_{(\frac{(p_0)'}{p'})'}(\rn)$.
\item[\rm(vi)] Let $\omega\in A_p(\rn)$ with some $p\in[1,\fz)$ and $\gz\in(0,1)$.
Then $\omega^{\gz}\in RH_{\gz^{-1}}(\rn)$ and there exists a positive constant
$C$, depending only on $[\omega]_{A_p(\rn)}$ and $\gz$, such that
$[\omega^\gz]_{RH_{\gz^{-1}}(\rn)}\le C$.
\item[\rm(vii)] Let $q_2:=q(1+\frac{1}{\gamma})$ with
$\gamma$ as in (ii) and let $q_1$ be as in (i).
Then $L^{q_2}(\boz)\subset L^q_\omega(\boz)\subset
L^{\frac q{q_1}}(\boz)$.
\end{enumerate}
\end{lemma}

Furthermore, we also need the  following
Lemma \ref{l3.3}, whose proof is similar to that of \cite[Lemma 4.38]{bm16},
and we omit the details here.

\begin{lemma}\label{l3.3}
Let $n\ge2$, $\boz\subset\rn$ be a bounded $\mathrm{NTA}$ domain, $0<p_0<q\le\fz$, and
$r_0\in(0,\diam(\boz))$. Assume that $x\in\overline{\boz}$ and
the weak reverse H\"older inequality
\begin{equation*}
\lf[\fint_{B_\boz(x,r)}|g|^q\,dx\r]^{\frac1q}
\le C_3\lf[\fint_{B_\boz(x,2r)}
|g|^{p_0}\,dx\r]^{\frac1{p_0}}
\end{equation*}
holds true for a given measurable function $g$ on $\boz$ and any $r\in(0,r_0)$,
where $C_3$ is a positive constant, independent of $x$ and $r$,
which may depend on $g$.
Then, for any given $p\in(0,\fz]$, there exists a positive constant $C$,
depending only on $p$, $p_0$, $q$, and $C_3$, such that
\begin{equation*}
\lf[\fint_{B_\boz(x,r)}|g|^q\,dx\r]^{\frac1q}
\le C\lf[\fint_{B_\boz(x,2r)}
|g|^{p}\,dx\r]^{\frac1p}
\end{equation*}
holds true for any $r\in(0,r_0)$.
\end{lemma}

Now, we prove Theorem \ref{t1.1} by using Theorem \ref{t3.1} and Lemmas
\ref{l3.1}, \ref{l3.2}, and \ref{l3.3}.

\begin{proof}[Proof of Theorem \ref{t1.1}]
We first show (i).
Let $B:=B(x_B,r_B)\subset\rn$ be a ball satisfying $r_B\in(0,r_0/4)$ and either
$2B\subset\boz$ or $x_B\in\partial\boz$.
Take $\phi\in C^\fz_\mathrm{c}(\rn)$ such that $\phi\equiv1$ on $2B$, $0\le\phi\le1$,
and $\supp(\phi)\subset 4B$. Let $w,\ v\in W^{1,2}_0(\boz)$ be respectively the
weak solutions of the Dirichlet problems
\begin{equation}\label{3.6}
\begin{cases}
-\mathrm{div}(A\nabla w)=\mathrm{div}(\phi\mathbf{f})\ \ &\text{in}\ \ \boz,\\
w=0 \ \ &\text{on}\ \ \partial\boz
\end{cases}
\end{equation}
and
\begin{equation}\label{3.7}
\begin{cases}
-\mathrm{div}(A\nabla v)=\mathrm{div}((1-\phi)\mathbf{f})\ \ &\text{in}\ \ \boz,\\
v=0 \ \ &\text{on}\ \ \partial\boz.
\end{cases}
\end{equation}
Then $u=w+v$ and $\nabla u=\nabla w+\nabla v$.
Let $F:=|\nabla u|$, $f:=|\mathbf{f}|$, $F_B:=|\nabla w|$, and
$R_B:=|\nabla v|$. It is easy to see that $0\le F\le F_B+R_B$.
By $\mathbf{f}\in L^2(\boz;\rn)$, \eqref{3.6}, and the fact that \eqref{1.9}
holds true for $p=2$ (see Remark \ref{r1.2}), we conclude that
\begin{align}\label{3.8}
\fint_{2B_\boz}F_B^2\,dx=
\fint_{2B_\boz}|\nabla w|^2\,dx
\ls\frac{1}{|2B_\boz|}\int_{\boz}|\mathbf{f}\phi|^2\,dx\ls\fint_{4B_\boz}|\mathbf{f}|^2\,dx
\sim\fint_{4B_\boz}f^2\,dx.
\end{align}
Moreover, from \eqref{3.7} and the assumption \eqref{1.8} of this theorem,
it follows that \eqref{1.8} holds true for the above $v$,
which, together with the self-improvement property of the weak reverse H\"older inequality
(see, for instance, \cite[pp.\,122-123]{g83}), further implies that
there exists an $\uc_0\in(0,\fz)$ such that the inequality \eqref{1.8}
holds true with $p$ replaced by $p+\uc_0$.
By this and Lemma \ref{l3.3}, we conclude that,
for any $q\in(0,2]$, the weak reverse H\"older inequality
\begin{equation}\label{3.9}
\lf(\fint_{B_\boz}|\nabla v|^{p+\uc_0}\,dx\r)^{1/(p+\uc_0)}
\ls\lf(\fint_{2B_\boz}
|\nabla v|^{q}\,dx\r)^{1/q}
\end{equation}
holds true, which, combined with \eqref{3.8}, further implies that
\begin{align}\label{3.10}
\lf(\fint_{B_\boz}R_B^{p+\uc_0}\,dx
\r)^{1/(p+\uc_0)}&\ls\lf(\fint_{2B_\boz}|\nabla v|^2\,dx\r)^{1/2}
\ls\lf(\fint_{4B_\boz}|\nabla u|^2\,dx\r)^{1/2}
+\lf(\fint_{4B_\boz}f^2\,dx\r)^{1/2}\\ \nonumber
&\sim\lf(\fint_{4B_\boz}F^2\,dx\r)^{1/2}
+\lf(\fint_{4B_\boz}f^2\,dx\r)^{1/2}.
\end{align}
From \eqref{3.8} and \eqref{3.10}, we deduce that \eqref{3.1} and \eqref{3.2}
hold true with $p_2:=p+\uc_0$ and $p_1:=2$. Thus, by Theorem \ref{t3.1} with $\omega\equiv1$
and $q:=p$, and the H\"older inequality, we conclude that
\begin{align*}
\lf(\fint_{\boz}|\nabla u|^p\,dx\r)^{1/p}&\ls\lf(\fint_{\boz}|\nabla u|^2\,dx\r)^{1/2}+
\lf(\fint_{\boz}|\mathbf{f}|^p\,dx\r)^{1/p}\\
&\ls\lf(\fint_{\boz}|\mathbf{f}|^2\,dx\r)^{1/2}+
\lf(\fint_{\boz}|\mathbf{f}|^p\,dx\r)^{1/p}
\ls\lf(\fint_{\boz}|\mathbf{f}|^p\,dx\r)^{1/p},
\end{align*}
which implies that \eqref{1.9} holds true.
Thus, (i) holds true.

Next, we prove (ii). Let $q\in[2,p]$, $q_0\in[1,\frac{q}{p'}]$, $r_0\in[(\frac{p}{q})',\fz]$,
and $\omega\in A_{q_0}(\rn)\cap RH_{r_0}(\rn)$. Assume that
$\uc_0\in(0,\fz)$ is as in \eqref{3.9}. Then $\omega\in A_{\frac{q}{(p+\uc_0)'}}(\rn)\cap RH_s(\rn)$
with $s\in((\frac{p+\uc_0}{q})',\fz]$.
Let $u$ be the weak solution of the Dirichlet problem $(D)_{q,\,\omega}$ with
$\mathbf{f}\in L^q_\omega(\boz;\rn)$. Then, from Lemma \ref{l3.2}(vii),
it follows that
\begin{equation}\label{3.11}
L^q_\omega(\boz)\subset L^{q/q_0}(\boz).
\end{equation}
By the fact that $q_0\le q/p'<q/(p+\uc_0)'$, we find that
$(p+\uc_0)'<q/q_0$, which, together with the H\"older inequality
and the assumption that $\boz$ is bounded, implies that
$L^{q/q_0}(\boz)\subset L^{(p+\uc_0)'}(\boz)$. From this and \eqref{3.11}, we deduce that
$\mathbf{f}\in L^{q/q_0}(\boz;\rn)\subset L^{(p+\uc_0)'}(\boz;\rn)$.

Let $B:=B(x_B,r_B)\subset\rn$ be a ball satisfying $|B|\le\bz_1|\boz|$ and either $2B\subset\boz$ or $x_B\in\partial\boz$,
where $\bz_1\in(0,1)$ is as in Theorem \ref{t3.1}.
Take $\phi\in C^\fz_\mathrm{c}(\rn)$ such that $\phi\equiv1$ on $2B$, $0\le\phi\le1$,
and $\supp(\phi)\subset 4B$.
Let $w$ and $v$ be, respectively, as in \eqref{3.6} and \eqref{3.7}.
Then  $u=w+v$ and $\nabla u=\nabla w+\nabla v$.
Recall that $F:=|\nabla u|$, $f:=|\mathbf{f}|$, $F_B:=|\nabla w|$,
and $R_B:=|\nabla v|$.
By the proof of (i), we know that \eqref{1.9} holds true
with $p$ replaced by $p+\uc_0$, which, combined with Lemma \ref{l3.1} and \eqref{3.6},
further implies that
$$\|\nabla w\|_{L^{(p+\uc_0)'}(\boz;\rn)}\ls\lf\|\phi\mathbf{f}\r\|_{L^{(p+\uc_0)'}(\boz;\rn)}
\ls\lf\|\mathbf{f}\r\|_{L^{(p+\uc_0)'}(4B_\boz;\rn)}.$$
From this, it follows that
\begin{align}\label{3.12}
\fint_{2B_\boz}F_B^{(p+\uc_0)'}\,dx
\ls\fint_{4B_\boz}f^{(p+\uc_0)'}\,dx.
\end{align}
Furthermore, by the assumption \eqref{1.8} of this theorem and \eqref{3.7},
we conclude that \eqref{1.8} holds true for the above $v$.
Thus, \eqref{3.10} holds true for the above $v$, which,
together with \eqref{3.12}, further implies that
\begin{align}\label{3.13}
\lf[\fint_{B_\boz}R_B^{p+\uc_0}
\,dx\r]^{1/(p+\uc_0)}&=\lf[\fint_{B_\boz}|\nabla v|^{p+\uc_0}
\,dx\r]^{1/(p+\uc_0)}\ls\lf[\fint_{2B_\boz}|\nabla v|^{(p+\uc_0)'}\,dx\r]^{1/(p+\uc_0)'}\\ \nonumber
&\ls\lf[\fint_{2B_\boz}|\nabla u|^{(p+\uc_0)'}\,dx\r]^{1/(p+\uc_0)'}
+\lf[\fint_{4B_\boz}|\mathbf{f}|^{(p+\uc_0)'}\,dx\r]^{1/(p+\uc_0)'}\\ \nonumber
&\ls\lf[\fint_{4B_\boz}F^{(p+\uc_0)'}\,dx\r]^{1/(p+\uc_0)'}
+\lf[\fint_{4B_\boz}f^{(p+\uc_0)'}\,dx\r]^{1/(p+\uc_0)'}.
\end{align}
From \eqref{3.12} and \eqref{3.13}, we deduce that \eqref{3.1} and \eqref{3.2}
hold true with $p_2:=p+\uc_0$ and $p_1:=(p+\uc_0)'$, which, combined with
$q<p+\uc_0$, Theorem \ref{t3.1}, and \eqref{3.11}, further implies that
\begin{align*}
\lf[\frac{1}{\omega(\boz)}\int_{\boz}|\nabla u|^q\omega\,dx\r]^{1/q}&
=\lf[\frac{1}{\omega(\boz)}\int_{\boz}F^q\omega\,dx\r]^{1/q}\\
&\ls\lf[\frac{1}{|\boz|}\int_{\boz}|F|^{(p+\uc_0)'}\,dx\r]^{1/(p+\uc_0)'}+
\lf[\frac{1}{\omega(\boz)}\int_{\boz}|\mathbf{f}|^q\omega\,dx\r]^{1/q}\\
&\sim\lf[\frac{1}{|\boz|}\int_{\boz}|\nabla u|^{(p+\uc_0)'}\,dx\r]^{1/(p+\uc_0)'}+
\lf[\frac{1}{\omega(\boz)}\int_{\boz}|\mathbf{f}|^q\omega\,dx\r]^{1/q}\\
&\ls\lf[\frac{1}{|\boz|}\int_{\boz}|\mathbf{f}|^{(p+\uc_0)'}\,dx\r]^{1/(p+\uc_0)'}+
\lf[\frac{1}{\omega(\boz)}\int_{\boz}|\mathbf{f}|^q\omega\,dx\r]^{1/q}\\
&\ls\lf[\frac{1}{\omega(\boz)}\int_{\boz}|\mathbf{f}|^q\omega\,dx\r]^{1/q}.
\end{align*}
Therefore, \eqref{1.10} holds true, which shows (ii).
This finishes the proof of Theorem \ref{t1.1}.
\end{proof}

\section{Proof of Theorem \ref{t1.2}\label{s4}}

 In this section, we prove Theorem \ref{t1.2}
by using Theorem \ref{t1.1} and the method of perturbation.
We begin with the following reverse H\"older inequality established in
\cite[Lemma 3.2 and Corollary 4.1]{lp19}.

\begin{lemma}\label{l4.1}
Let $\boz\subset\rn$ be a bounded $\mathrm{NTA}$ domain, $B(x_0,r)$ a ball such that $r\in(0,r_0/4)$
and either $x_0\in\partial\boz$ or $B(x_0,2r)\subset\boz$, where $r_0\in(0,\diam(\boz))$ is a constant.
Assume that the matrix $A:=a+b$ satisfies Assumption \ref{a1} and
$u\in W^{1,2}(B_\boz(x_0,2r))$ is the weak solution of the following Dirichlet problem
\begin{equation*}
\begin{cases}
-\mathrm{div}(A\nabla u)=0\ \ & \text{in}\ \ B_\boz(x_0,2r),\\
u=0 \ \ & \text{on}\ \ B(x_0,2r)\cap\partial\boz.
\end{cases}
\end{equation*}
Then there exists a positive constant $p\in(2,\fz)$,
depending on $\boz$, $n$, and $\mu_0$ in \eqref{1.3}, such that
$$\lf[\fint_{B_\boz(x_0,r)}|\nabla u|^p\,dx\r]^{1/p}\le C\lf[1+\|b\|_{\mathrm{BMO}(\boz)}\r]
\lf[\fint_{B_\boz(x_0,2r)}|\nabla u|^2\,dx\r]^{1/2},
$$
where $C$ is a positive constant depending only on $n$, $\boz$, and $p$.
\end{lemma}

To show Theorem \ref{t1.2}, we need the following perturbation argument which is motivated by \cite{cp98}.

\begin{lemma}\label{l4.2}
Let $n\ge2$, $\boz\subset\rn$ be a bounded $\mathrm{NTA}$ domain,
and $r_0\in(0,\diam(\boz))$ a constant. Assume that the matrix $A$ satisfies Assumption \ref{a1}.
Let $u\in W^{1,2}(B_\boz(x_0,4r))$ be a solution of the equation
$\mathrm{div}(A\nabla u)=0$ in $B_\boz(x_0,4r)$ with $u=0$
on $B(x_0,4r)\cap\partial\boz$, where $x_0\in\overline{\boz}$ and $r\in(0,r_0/4)$.
Then there exist a function $\tz:=\tz(r)$, $p\in(2,\fz)$,
and a function $v\in W^{1,p}(B_\boz(x_0,r))$ such that
\begin{equation}\label{4.1}
\lf[\fint_{B_\boz(x_0,r)}|\nabla v|^p\,dx\r]^{1/p}
\le C\lf[\fint_{B_\boz(x_0,4r)}|\nabla u|^2\,dx\r]^{1/2}
\end{equation}
and
\begin{align}\label{4.2}
\lf[\fint_{B_\boz(x_0,r)}|\nabla(u-v)|^2\,dx\r]^{1/2}
\le\tz(r)\lf[\fint_{B_\boz(x_0,4r)}|\nabla u|^2\,dx\r]^{1/2},
\end{align}
where $C$ is a positive constant independent of $u,\,v$, $x_0$, and $r$.
\end{lemma}

\begin{proof}
By the assumption that $A$ satisfies Assumption \ref{a1}, we find that $A=a+b$, where
$a:=\{a_{ij}\}_{i,j=1}^n$ is real-valued, symmetric, and measurable, and satisfies \eqref{1.3}, and
$b:=\{b_{ij}\}_{i,j=1}^n$ is real-valued, anti-symmetric, and measurable, and satisfies
$b_{ij}\in\mathrm{BMO}(\boz)$ for any $i,\,j\in\{1,\,\ldots,\,n\}$.

Let $a_0:=\{c_{ij}\}_{i,j=1}^n$, where, for any $i,\,j\in\{1,\,\ldots,\,n\}$,
$$c_{ij}:=\fint_{B_\boz(x_0,2r)}a_{ij}\,dx.$$
Assume that $v\in W^{1,2}(B_\boz(x_0,2r))$
is the solution of the following boundary value problem
\begin{equation}\label{4.3}
\begin{cases}
\mathrm{div}(a_0\nabla v)=0\ \ &\text{in}\ \ B_\boz(x_0,2r),\\
v=u\ \ &\text{on} \ \
\partial B_\boz(x_0,2r).
\end{cases}
\end{equation}
Then, from \eqref{4.3} and the fact that $u-v\in W^{1,2}_0(B_\boz(x_0,2r))$, we deduce that
\begin{align}\label{4.4}
\int_{B_\boz(x_0,2r)}a_0\nabla(u-v)\cdot\nabla(u-v)\,dx
=\int_{B_\boz(x_0,2r)}(a_0-A)\nabla u\cdot\nabla(u-v)\,dx.
\end{align}
Denote by $\wz{u}$ the $W^{1,2}$-extension of $u$ to $\rn$;
namely, $\wz{u}\in W^{1,2}(\rn)$ and $\wz{u}|_{B_\boz(x_0,2r)}=u$.
Furthermore, denote by $\wz{u-v}$ the zero extension
of $(u-v)|_{B_\boz(x_0,2r)}$ to $\rn$. Obviously, $\wz{u-v}\in W^{1,2}(\rn)$
and $\supp(\wz{u-v})\subset B_\boz(x_0,2r)$,
which, combined with the divergence theorem and the assumption that
$$b_{B_\boz(x_0,2r)}:=\lf\{\fint_{B_\boz(x_0,2r)}b_{ij}\,dx\r\}_{i,j=1}^n$$
is anti-symmetric, further implies that
\begin{align*}
\int_{B_\boz(x_0,2r)}b_{B_\boz(x_0,2r)}\nabla u\cdot\nabla(u-v)\,dx
&=\int_{B(x_0,2r)}b_{B_\boz(x_0,2r)}\nabla \wz{u}\cdot\nabla (\wz{u-v})\,dx\\
&=\int_{B(x_0,2r)}\mathrm{div}\lf(b_{B_\boz(x_0,2r)}
\nabla \wz{u}\r)(\wz{u-v})\,dx\\
&\quad-\int_{\partial B(x_0,2r)}
b_{B_\boz(x_0,2r)}\nabla \wz{u}\cdot\boldsymbol{\nu}(\wz{u-v})\,d\sz(x)\\
&=0,
\end{align*}
where $\boldsymbol{\nu}$ denotes the \emph{outward unit normal} to
$\partial B(x_0,2r)$.  By this, \eqref{4.4}, and the definitions of $a_0$ and
$b_{B_\boz(x_0,2r)}$, we find that
\begin{align}\label{4.5}
\int_{B_\boz(x_0,2r)}a_0\nabla(u-v)\cdot\nabla(u-v)\,dx
&=\int_{B_\boz(x_0,2r)}\lf(a_0-A+b_{B_\boz(x_0,2r)}\r)
\nabla u\cdot\nabla(u-v)\,dx\\ \nonumber
&=\int_{B_\boz(x_0,2r)}\lf(A_{B_\boz(x_0,2r)}-A\r)\nabla u\cdot\nabla(u-v)\,dx,
\end{align}
where
$$A_{B_\boz(x_0,2r)}:=\lf\{\fint_{B_\boz(x_0,2r)}(a_{ij}+b_{ij})\,dx\r\}_{i,j=1}^n.$$
Therefore, from \eqref{4.5}, \eqref{1.3}, and the H\"older inequality,
it follows that,
for any given $\uc\in(0,\fz)$, there exists a positive constant $C_{(\uc)}$,
depending on $\uc$, such that
\begin{align}\label{4.6}
&\mu_0\int_{B_\boz(x_0,2r)}|\nabla(u-v)|^2\,dx\\ \nonumber
&\quad\le
\int_{B_\boz(x_0,2r)}\lf|A-A_{B_\boz(x_0,2r)}\r||\nabla u||\nabla(u-v)|\,dx\\ \nonumber
&\quad\le\uc\int_{B_\boz(x_0,2r)}|\nabla(u-v)|^2\,dx+ C_{(\uc)}
\int_{B_\boz(x_0,2r)}\lf|A-A_{B_\boz(x_0,2r)}\r|^2|\nabla u|^2\,dx,
\end{align}
where $\mu_0\in(0,1)$ is as in \eqref{1.3}.
Take $\uc:=\mu_0/2$ in \eqref{4.6}. Then, by \eqref{4.6}, we conclude that
\begin{align}\label{4.7}
\lf[\int_{B_\boz(x_0,2r)}|\nabla(u-v)|^2\,dx\r]^{1/2}\le C_4
\lf[\int_{B_\boz(x_0,2r)}\lf|A-A_{B_\boz(x_0,2r)}\r|^2|\nabla u|^2\,dx\r]^{1/2}.
\end{align}

Moreover, from Lemma \ref{l4.1}, we deduce that there exist positive
constants $\widetilde{p}\in(1,\fz)$ and $C_5\in(0,\fz)$, independent of $u$, $x_0$, and $r$, such that
\begin{align}\label{4.8}
\lf[\fint_{B_\boz(x_0,2r)}
|\nabla u|^{2\widetilde{p}}\,dx\r]^{1/(2\widetilde{p})}\le
C_5\lf[1+\|b\|_{\mathrm{BMO}(\boz)}\r]
\lf[\fint_{B_\boz(x_0,4r)}
|\nabla u|^{2}\,dx\r]^{1/2}.
\end{align}
Furthermore, by \cite[Theorem 1]{j80}, we know that
there exist an $\wz{A}\in\mathrm{BMO}(\rn;\rr^{n^2})$
and a positive constant $C$, independent of $x_0$ and $r$,
such that
\begin{equation}\label{4.9}
\wz{A}|_{B_\boz(x_0,2r)}=A\quad \text{and}\quad
\|\wz{A}\|_{\mathrm{BMO}(\rn;\rr^{n^2})}\le
C\|A\|_{\mathrm{BMO}(B_\boz(x_0,2r);\rr^{n^2})}.
\end{equation}
From the H\"older inequality and the well-known
John--Nirenberg inequality on $\mathrm{BMO}(\rn)$
(see, for instance, \cite{g14a,St93}), it follows that
\begin{align*}
&\lf[\fint_{B_\boz(x_0,2r)}
\lf|A-A_{B_\boz(x_0,2r)}\r|^{2\widetilde{p}'}\,dx\r]^{1/(2\widetilde{p}')}\\
&\quad\ls\lf[\fint_{B(x_0,2r)}
\lf|\wz{A}-\wz{A}_{B(x_0,2r)}\r|^{2\widetilde{p}'}\,dx\r]^{1/(2\widetilde{p}')}
+\lf[\fint_{B(x_0,2r)}
\lf|\wz{A}_{B(x_0,2r)}-A_{B_\boz(x_0,2r)}
\r|^{2\widetilde{p}'}\,dx\r]^{1/(2\widetilde{p}')}\\
&\quad\ls\|\wz{A}\|_{\mathrm{BMO}(\rn;\rr^{n^2})}
+\lf|\wz{A}_{B(x_0,2r)}-A_{B_\boz(x_0,2r)}\r|\\
&\quad\ls\|\wz{A}\|_{\mathrm{BMO}(\rn;\rr^{n^2})}
+\fint_{B_\boz(x_0,2r)}\lf|A-\wz{A}_{B(x_0,2r)}\r|\,dx\\
&\quad\ls\|\wz{A}\|_{\mathrm{BMO}(\rn;\rr^{n^2})}
+\fint_{B(x_0,2r)}\lf|\wz{A}-\wz{A}_{B(x_0,2r)}\r|\,dx
\ls\|\wz{A}\|_{\mathrm{BMO}(\rn;\rr^{n^2})},
\end{align*}
where $\widetilde{p}'\in(1,\fz)$ is given by
$1/\widetilde{p}+1/\widetilde{p}'=1$, which,
together with \eqref{4.9}, further implies that
there exists a positive constant $C_6$, independent of $x_0$, $r$, and $A$,
such that
$$\lf[\fint_{B_\boz(x_0,2r)}
\lf|A-A_{B_\boz(x_0,2r)}\r|^{2\widetilde{p}'}\,dx\r]^{1/(2\widetilde{p}')}
\le C_6\|A\|_{\mathrm{BMO}(B_\boz(x_0,2r);\rr^{n^2})}.
$$
By this, \eqref{4.7} and \eqref{4.8}, we conclude that
\begin{align}\label{4.10}
&\lf[\fint_{B_\boz(x_0,2r)}
|\nabla (u-v)|^2\,dx\r]^{1/2}\\ \nonumber
&\quad\le C_4\lf[\fint_{B_\boz(x_0,2r)}
\lf|A-A_{B_\boz(x_0,2r)}\r|^{2\widetilde{p}'}\,dx\r]^{1/(2\widetilde{p}')}
\lf[\fint_{B_\boz(x_0,2r)}
|\nabla u|^{2\widetilde{p}}\,dx\r]^{1/(2\widetilde{p})}\\ \nonumber
&\quad\le C_4C_5C_6\lf[1+\|b\|_{\mathrm{BMO}(\boz)}\r]
\|A\|_{\mathrm{BMO}(B_\boz(x_0,2r);\rr^{n^2})}\lf[\fint_{B_\boz(x_0,4r)}
|\nabla u|^{2}\,dx\r]^{1/2}\\ \nonumber
&\quad\le\tz(r)\lf[\fint_{B_\boz(x_0,4r)}
|\nabla u|^{2}\,dx\r]^{1/2},
\end{align}
where
\begin{equation}\label{4.11}
\tz(r):=C_4C_5C_6\lf[1+\|b\|_{\mathrm{BMO}(\boz)}\r]
\sup_{B(x,2t)\subset\boz,\,t\in(0,r]}\fint_{B(x,2t)}
\lf|A-\fint_{B(x,2t)}A\,dz\r|\,dy.
\end{equation}
Thus, \eqref{4.2} holds true.
Furthermore, from the known regularity theory of
second order elliptic equations
(see, for instance, \cite[Chapter 1]{k94} and \cite[Chapter V]{g83}),
we deduce that there exists a $p\in(2,\fz)$ such that
$$\lf[\fint_{B_\boz(x_0,2r)}
|\nabla v|^p\,dx\r]^{1/p}\ls\lf[\fint_{B_\boz(x_0,4r)}
|\nabla v|^2\,dx\r]^{1/2},
$$
which, combined with \eqref{4.10}, \eqref{4.11}, and the fact that
$A\in \mathrm{BMO}(\boz;\rr^{n^2})$,
further implies that \eqref{4.1} holds true.
This finishes the proof of Lemma \ref{l4.2}.
\end{proof}

Furthermore, to prove Theorem \ref{t1.2}, we need the following conclusion
for the constant coefficient boundary value problems, which was established in \cite[Lemma 4.1]{sh05a}
(see also \cite[Lemma 4.1]{g12}).

\begin{lemma}\label{l4.3}
Let $n\ge2$, $\boz\subset\rn$ be a bounded Lipschitz domain,
and $a_0:=\{a_{ij}\}_{i,j=1}^n$ a symmetric constant coefficient matrix
that satisfies \eqref{1.3}.
Assume that $v\in W^{1,2}(B_\boz(x_0,2r))$ is a weak solution
of the equation $\mathrm{div}(a_0\nabla v)=0$ in $B_\boz(x_0,2r)$ with
$v=0$ on $B(x_0,2r)\cap\partial\boz$, where $B(x_0,r)$ is a ball such that $r\in(0,r_0/4)$
and either $x_0\in\partial\boz$ or $B(x_0,2r)\subset\boz$, and $r_0\in(0,\diam(\boz))$
is a constant. Then the weak reverse H\"older inequality
\begin{equation*}
\lf[\fint_{B_\boz(x_0,r)}|\nabla v|^p\,dx\r]^{\frac1p}
\le C\lf[\fint_{B_\boz(x_0,2r)}|\nabla v|^2\,dx\r]^{\frac12}
\end{equation*}
holds true for $p:= 3+\uc$ when $n\ge3$, or $p:=4+\uc$ when $n=2$,
where  $C$ and $\uc$ are positive constants depending only on $n$, the Lipschitz constant of $\boz$,
and $\mu_0$ in \eqref{1.3}.
\end{lemma}

Now, we prove Theorem \ref{t1.2} via using Theorem \ref{t3.1} and Lemmas \ref{l4.2} and \ref{l4.3}.

\begin{proof}[Proof of Theorem \ref{t1.2}]
We first show (i). Let $A$ satisfy Assumption \ref{a1} and $u$ be the weak
solution of the following Dirichlet problem
\begin{equation}\label{4.12}
\begin{cases}
-\mathrm{div}(A\nabla u)=\mathrm{div}(\mathbf{f})\ \ &\text{in}\ \ \boz,\\
u=0 \ \ &\text{on}\ \ \partial\boz.
\end{cases}
\end{equation}
Based on Lemma \ref{l3.1}, to prove \eqref{1.11}, it suffices to show that
there exists a positive constant $\uc_0\in(0,\fz)$, depending only on $n$
and the Lipschitz constant of $\boz$, such that, for any given $p\in[2,3+\uc_0)$
when $n\ge3$, or $p\in[2,4+\uc_0)$ when $n=2$,
there exists a $\dz_0\in(0,\fz)$, depending only on $n$, $p$, and the Lipschitz constant of $\boz$,
such that, if $A$ satisfies the $(\dz,R)$-$\mathrm{BMO}$ condition for some $\dz\in(0,\dz_0)$
and $R\in(0,\fz)$, or $A\in\mathrm{VMO}(\boz)$, then \eqref{1.11} holds true.

By the definition of $\mathrm{VMO}(\boz)$,
we conclude that, if $A\in\mathrm{VMO}(\boz)$, then there exists an $r_1\in(0,\fz)$ such that,
for any $r\in(0,r_1)$, $\tz(r)<\uc_0/2$, where $\tz(r)$ and $\uc_0$ are, respectively, as in \eqref{4.11} and Theorem \ref{t3.1}.
Let $B(x_0,r)\subset\rn$ be such
that $r\in(0,\min\{r_0,r_1\}/4)$ and either $x_0\in\partial\boz$ or $B(x_0,2r)\subset\boz$,
where $r_0\in(0,\diam(\boz))$ is as in Lemma \ref{l4.3}.
Assume that $v\in W^{1,2}(B_\boz(x_0,2r))$ is a weak solution of the equation
$\mathrm{div}(A\nabla v)=0$ in $B_\boz(x_0,2r)$ with
$v=0$ on $B(x_0,2r)\cap\partial\boz$. Let $w\in W^{1,2}(B_\boz(x_0,2r))$ be a weak solution
of the equation $\mathrm{div}(a_0\nabla w)=0$ in $B_\boz(x_0,2r)$ with
$w=v$ on $\partial B_\boz(x_0,2r)$, where $a_0:=\{c_{ij}\}_{i,j=1}^n$ with
$c_{ij}:=\fint_{B_\boz(x_0,2r)}a_{ij}\,dx$ for any $i,\,j\in\{1,\,\ldots,\,n\}$.
Then, from Lemmas \ref{l4.2} and \ref{l4.3}, it follows that
\begin{equation}\label{4.13}
\lf[\fint_{B_\boz(x_0,r)}|\nabla w|^p\,dx\r]^{\frac1{p}}
\ls\lf[\fint_{B_\boz(x_0,2r)}|\nabla v|^2\,dx\r]^{\frac12}
\end{equation}
and
\begin{align}\label{4.14}
\lf[\fint_{B_\boz(x_0,r)}|\nabla(v-w)|^2\,dx\r]^{\frac12}
\le\tz(r)\lf[\fint_{B_\boz(x_0,2r)}|\nabla v|^2\,dx\r]^{\frac12},
\end{align}
where $\tz(r)$ is as in \eqref{4.11}.
By the well-known John--Nirenberg inequality on $\mathrm{BMO}(\boz)$
(see, for instance, \cite{g14a,St93}), we conclude that
there exists a $\dz_0\in(0,\fz)$ sufficiently small such that, if
$A$ satisfies the $(\dz,R)$-$\mathrm{BMO}$ condition for some $\dz\in(0,\dz_0)$
and $R\in(0,\fz)$, then, for any $r\in(0,r_0)$,
$\tz(r)<\uc_0/2$, where $\uc_0$ is as in Theorem \ref{t3.1}.
From this, \eqref{4.13}, \eqref{4.14}, $r\in(0,\min\{r_0,r_1\}/4)$,
and Theorem \ref{t3.1} with $\omega\equiv1$, we deduce that,
if $A$ satisfies the $(\dz,R)$-$\mathrm{BMO}$ condition for some $\dz\in(0,\dz_0)$
and $R\in(0,\fz)$, or $A\in\mathrm{VMO}(\boz)$, then
$|\nabla v|\in L^p(B_\boz(x_0,2r))$ and
\begin{equation}\label{4.15}
\lf[\fint_{B_\boz(x_0,r)}|\nabla v|^{p}\,dx\r]^{\frac1p}
\ls\lf[\fint_{B_\boz(x_0,2r)}|\nabla v|^2\,dx\r]^{\frac12},
\end{equation}
which, together with Theorem \ref{t1.1}
and the fact that the Dirichlet problem $(D)_2$ is uniquely solvable,
further implies that \eqref{1.11} holds true in the case $p\in[2,3+\uc_0)$ when $n\ge3$,
or $p\in[2,4+\uc_0)$ when $n=2$. This finishes the proof of (i).

Now, we prove (ii). Let $\uc_0$ be as in (i) and $u$ the weak solution of the
Dirichlet problem \eqref{4.12} with $\mathbf{f}\in L^p_\omega(\boz;\rn)$.
We first assume that $p\in[2,p_0)$
and $\omega\in A_{\frac{p}{p_0'}}(\rn)\cap RH_{(\frac{p_0}{p})'}(\rn)$,
where $p_0:=3+\uc_0$ when $n\ge3$, or $p_0:=4+\uc_0$ when $n=2$.
Then, by \eqref{4.15} and Theorem \ref{t1.1}, we conclude that
there exists a $\dz_0\in(0,\fz)$, depending only on $n$, $p$,
$[\omega]_{A_{\frac{p}{p_0'}}(\rn)}$,
$[\omega]_{RH_{(\frac{p_0}{p})'}(\rn)}$, and the Lipschitz constant of $\boz$,
such that, if $A$ satisfies the $(\dz,R)$-$\mathrm{BMO}$
condition for some $\dz\in(0,\dz_0)$
and $R\in(0,\fz)$, or $A\in\mathrm{VMO}(\boz)$, then \eqref{1.12} holds true.

Next, we assume that $p\in(p_0',2)$
and $\omega\in A_{\frac{p}{p_0'}}(\rn)\cap RH_{(\frac{p_0}{p})'}(\rn)$.
Then, from Lemma \ref{l3.2}(v), we deduce that $p'\in(2,p_0)$
and $\omega^{1-p'}\in A_{\frac{p'}{p_0'}}(\rn)\cap RH_{(\frac{p_0}{p'})'}(\rn)$.
Let $\mathbf{g}\in L^{p'}_{\omega^{1-p'}}(\boz;\rn)$ and $v$ be the weak
solution of the following Dirichlet problem
\begin{equation*}
\begin{cases}
-\mathrm{div}(A^\ast\nabla v)=\mathrm{div}(\mathbf{g})\ \ &\text{in}\ \ \boz,\\
v=0 \ \ &\text{on}\ \ \partial\boz,
\end{cases}
\end{equation*}
where $A^\ast$ denotes the transpose of $A$.
By the assumption that $A$ satisfies Assumption \ref{a1}, we find that
$A^\ast$ also satisfies Assumption \ref{a1}. Thus, we have
\begin{equation}\label{4.16}
\|\nabla v\|_{L^{p'}_{\omega^{1-p'}}(\boz;\rn)}\ls\|\mathbf{g}\|_{L^{p'}_{\omega^{1-p'}}(\boz;\rn)}.
\end{equation}
Moreover,
$$\int_\boz\mathbf{g}\cdot \nabla u\,dx=-\int_\boz A^\ast\nabla v\cdot\nabla u\,dx
=-\int_\boz A\nabla u\cdot\nabla v\,dx=
\int_\boz \mathbf{f}\cdot\nabla v\,dx,
$$
which, combined with \eqref{4.16} and the H\"older inequality, further implies that
\begin{align}\label{4.17}
\|\nabla u\|_{L^{p}_\omega(\boz;\rn)}&=\sup_{\|\mathbf{g}
\|_{L^{p'}_{\omega^{1-p'}}(\boz;\rn)}\le1}
\lf|\int_\boz\mathbf{g}\cdot\nabla u\,dx\r|
=\sup_{\|\mathbf{g}\|_{L^{p'}_{\omega^{1-p'}}(\boz;\rn)}\le1}
\lf|\int_\boz\mathbf{f}\cdot \nabla v\,dx\r|\\
&\le\sup_{\|\mathbf{g}\|_{L^{p'}_{\omega^{1-p'}}(\boz;\rn)}\le1}\|\mathbf{f}\|_{L^{p}_\omega(\boz;\rn)}
\|\nabla v\|_{L^{p'}_{\omega^{1-p'}}(\boz;\rn)}
\nonumber\\
&\ls\sup_{\|\mathbf{g}\|_{L^{p'}_{\omega^{1-p'}}(\boz;\rn)}\le1}
\|\mathbf{f}\|_{L^{p}_\omega(\boz;\rn)}
\|\mathbf{g}\|_{L^{p'}_{\omega^{1-p'}}(\boz;\rn)}
\ls\|\mathbf{f}\|_{L^{p}_\omega(\boz;\rn)}.\nonumber
\end{align}
From this, it follows that \eqref{1.12} holds true when $p\in(p_0',2)$.
This finishes the proof of (ii) and hence of Theorem \ref{t1.2}.
\end{proof}

\section{Proof of Theorem \ref{t1.3}\label{s5}}

In this section, we give the proofs of Theorem \ref{t1.3}
and Corollary \ref{c1.1} by using Theorem \ref{t1.1} and some properties of
quasi-convex domains. We begin with the following Lemma \ref{l5.1}.

\begin{lemma}\label{l5.1}
Let $n\ge2$ and $\boz\subset\rn$ be a bounded $\mathrm{NTA}$ domain.
Assume that $p\in(2,\fz)$, $\boz$ is a
$(\dz,\sz,R)$ quasi-convex domain for some $\dz,\ \sz\in(0,1)$
and $R\in(0,\fz)$, $a_0:=\{a_{ij}\}_{i,j=1}^n$ a symmetric
constant coefficient matrix that satisfies \eqref{1.3}.
Let $v\in W^{1,2}(B_\boz(x_0,2r))$ be a weak solution
of the equation $\mathrm{div}(a_0\nabla v)=0$ in $B_\boz(x_0,2r)$ with
$v=0$ on $B(x_0,2r)\cap\partial\boz$, where $B(x_0,r)$ is a ball such that
$r\in(0,r_0/4)$ and either $x_0\in\partial\boz$ or $B(x_0,2r)\subset\boz$,
and $r_0\in(0,R)$ is a constant. Then there exists a positive constant $\dz_0\in(0,1)$,
depending only on $n$, $p$, and $\boz$, such that, if $\boz$ is a
$(\dz,\sz,R)$ quasi-convex domain for some $\dz\in(0,\dz_0)$, $\sz\in(0,1)$,
and $R\in(0,\fz)$, then the weak reverse H\"older inequality
\begin{equation}\label{5.1}
\lf[\fint_{B_\boz(x_0,r)}|\nabla v|^p\,dx\r]^{\frac1p}
\le C\lf[\fint_{B_\boz(x_0,2r)}|\nabla v|^2\,dx\r]^{\frac12}
\end{equation}
holds true, where $C$ is a positive constant depending only on $n$, $\dz$, $\sz$,
$R$, and $\diam(\boz)$.
\end{lemma}

To show Lemma \ref{l5.1}, we need the following Lemma \ref{l5.2},
which is a special case of Theorem \ref{t3.1}.

\begin{lemma}\label{l5.2}
Let $n\ge2$, $\boz\subset\rn$ be a bounded $\mathrm{NTA}$ domain,
$p_1,\,p_2\in[1,\fz)$ with $p_2>p_1$,
$F\in L^{p_1}(\boz)$, and $q\in(p_1,p_2)$.
Suppose that, for any ball $B:=B(x_B,r_B)\subset\rn$ having the property
that $|B|\le \bz_1|\boz|$ and either $2B\subset\boz$ or $x_B\in\partial\boz$,
there exist two measurable functions $F_B$ and $R_B$ on $2B$
such that $|F|\le|F_B|+|R_B|$ on $2B\cap\boz$,
\begin{align*}\label{5.1e}
\lf(\fint_{2B_\boz}|R_B|^{p_2}\,dx\r)^{\frac1{p_2}}
\le C_7\lf(\fint_{\bz_2 B_\boz}|F|^{p_1}\,dx\r)^{\frac1{p_1}}
\end{align*}
and
\begin{align*}
\lf(\fint_{2B_\boz}|F_B|^{p_1}\,dx\r)^{\frac1{p_1}}
\le \uc\lf(\fint_{\bz_2 B_\boz}|F|^{p_1}\,dx\r)^{\frac1{p_1}},
\end{align*}
where $C_7,\,\uc$, and $\bz_1<1<\bz_2$ are positive
constants independent of $F,\,R_B,\,F_B$, and $B$.
Then there exists a positive constant $\uc_0$,
depending only on $C_7,\,n,\,p_1,\,p_2,\,q,\,\bz_1$, and $\bz_2$, such that,
if $\uc\in[0,\uc_0)$, then
\begin{align*}
\lf(\fint_{\boz}|F|^q\,dx\r)^{\frac1q}\le
C\lf(\fint_{\boz}|F|^{p_1}\,dx\r)^{\frac1{p_1}},
\end{align*}
where $C$ is a positive constant depending only on $C_7$,
$\uc_0$, $n$, $p_1,\,p_2$, $q$, $\bz_1$, and $\bz_2$.
\end{lemma}

Moreover, we also need the following Lemma \ref{l5.3}, which was established
in \cite[Lemmas 3.4 and 3.5]{z19}.

\begin{lemma}\label{l5.3}
Let $n\ge2$ and $\boz\subset\rn$ be a bounded $(\dz,\sz,R)$
quasi-convex domain for some $\dz,\ \sz\in(0,1)$ and $R\in(0,\fz)$.
Assume that $x_0\in\partial\boz$, $r\in(0,R/4)$, and $V_{4r}$
is the convex hull of $B_\boz(x_0,4r)$.
For any $t\in(0,1)$, let
\begin{equation}\label{5.2}
\boz_{tr}:=\{x\in\boz:\ \dist(x,\partial\boz)<tr\}
\end{equation}
and
\begin{equation}\label{5.3}
W_{r,t}:=\{x\in V_{4r}:\ \dist(x,\partial V_{4r}\cap B(x_0,3r))\le(t+\dz)r\}.
\end{equation}
\begin{itemize}
\item[\rm(i)] Then, for any $t\in(0,1)$, $\boz_{tr}\cap B(x_0,r)\subset W_{r,t}$.
Moreover, there exists a positive constant $C$, depending only on $n$
and $\sz$, such that $|W_{r,t}|\le C(t+\dz)r^n$.
\item[\rm(ii)] Let $u\in W^{1,2}(B(x_0,4r))$ satisfy $u=0$ on
$B(x_0,4r)\backslash V_{4r}$. Then, for any $t\in(0,1-\dz)$,
there exists a positive constant $C$, depending only on $n$ and $\sz$,
such that
$$\int_{B(x_0,r)\cap\boz_{tr}}u^2\,dx\le
C(t+\dz)^2r^2\int_{W_{r,t}}|\nabla u|^2\,dx.
$$
\end{itemize}
\end{lemma}

Now, we prove Lemma \ref{l5.1} via using Lemmas \ref{l4.1}, \ref{l5.2}, and \ref{l5.3}.

\begin{proof}[Proof of Lemma \ref{l5.1}]
We borrow some ideas from \cite{z19}. Observing that the matrix $a_0$ is symmetric and elliptic, and
has constant coefficients, without loss of generality, by a change of the coordinate system,
we may assume that $a_0=I$ (the unit matrix), namely, $\Delta v=0$ in $B(x_0,2r)\cap\boz$
and $v=0$ on $B(x_0,2r)\cap\partial\boz$. If $B(x_0,2r)\subset\boz$,
from the interior gradient estimates of harmonic functions (see, for instance, \cite[Theorem 2.10]{gt01}),
it follows that \eqref{5.1} holds true for any given $p\in(2,\fz)$.

Next, we assume that $x_0\in\partial\boz$.
By the assumption that $v=0$ on $B(x_0,2r)\cap\partial\boz$, we know that
$v$ can be extended to a function $\wz{v}\in W^{1,2}(B(x_0,2r))$ by zero-extension.
Let $s\in(0,r/16)$.
Since $\boz$ is a bounded $(\dz,\,\sz,\,R)$ quasi-convex domain,
from Definition \ref{d2.2} and Remark \ref{r2.2}(ii), it follows that
the convex hull of $B_\boz(x_0,s)$, denoted by $V_s$, satisfies that
$$V_s\cap\boz=B(x_0,s)\cap\boz\quad\quad\text{and}\quad\quad
d_H(\partial V_s,\partial(B(x_0,s)\cap\boz))\le \dz s.
$$
Let $w\in W^{1,2}(V_s)$ be the weak solution of the following Dirichlet problem
\begin{equation*}
\begin{cases}
-\Delta w=0\ \ & \text{in}\ \ V_s,\\
w=\wz{v} \ \ & \text{on}\ \ \partial V_s.
\end{cases}
\end{equation*}
By the assumptions that $\wz{v}=0$ on $\partial V_s\cap B(x_0,s)$ and that $V_s$
is convex, and the well known result for the boundary gradient estimates of harmonic functions
in convex domains (see, for instance, \cite[Theorem 1.1]{bl14}), we find that
\begin{equation}\label{5.4}
\|\nabla w\|_{L^\fz(B_\boz(x_0,s))}
\ls\lf[\fint_{B_\boz(x_0,2s)}|\nabla w|^2\,dx\r]^{1/2}.
\end{equation}
We now clam that there exists a positive constant $\epsilon$ such that
\begin{equation}\label{5.5}
\lf[\fint_{B_\boz(x_0,s)}|\nabla(v-w)|^2\,dx\r]^{1/2}\le C\dz^{\epsilon}
\lf[\fint_{B_\boz(x_0,2s)}|\nabla v|^2\,dx\r]^{1/2},
\end{equation}
where $C$ is a positive constant independent of $x_0$, $s$, $\dz$, and $v$.
If \eqref{5.5} holds true, then, from \eqref{5.5}, \eqref{5.4},
and Lemma \ref{l5.2}, we deduce that there exists a $\dz_0\in(0,1)$
such that, if $\boz$ is a $(\dz,\sz,R)$ quasi-convex domain
for some $\dz\in(0,\dz_0)$, $\sz\in(0,1)$, and $R\in(0,\fz)$,
then \eqref{5.1} holds true.

Next, we prove \eqref{5.5}. For any given $t\in(0,1)$, let
$\boz_{ts}$ be as in \eqref{5.2}.
Meanwhile, take $\tz_{\dz s}\in C^\fz_{\mathrm{c}}(\rn)$
satisfying that $0\le\tz_{\dz s}\le1$,
$\tz_{\dz s}\equiv1$ on $\boz\backslash\boz_{2\dz s}$, $\tz_{\dz s}=0$
in $\boz_{\dz s}$, and $|\nabla \tz_{\dz s}|\ls(\dz s)^{-1}$.
Then
\begin{align}\label{5.6}
\int_{V_s}|\nabla(w-\wz{v})|^2\,dx&=\int_{V_s}\nabla w\cdot\nabla(w-\wz{v})\,dx
-\int_{V_s}\nabla \wz{v}\cdot\nabla(w-\wz{v})\,dx\\ \nonumber
&=-\int_{V_s}\nabla \wz{v}\cdot\nabla(w-\wz{v})\,dx
=-\int_{B_\boz(x_0,s)}\nabla v\cdot\nabla(w-v)\,dx\\ \nonumber
&=-\int_{B_\boz(x_0,s)}\nabla(\tz_{\dz s}v)\cdot\nabla(w-v)\,dx\\ \nonumber
&\quad-\int_{B_\boz(x_0,s)}\nabla\lf((1-\tz_{\dz s})v\r)\cdot\nabla(w-v)\,dx.
\end{align}
By the assumptions that $\Delta v=0$ in $B_\boz(x_0,s)$, and that $v=0$ on
$B(x_0,2s)\cap\partial\boz$, and Lemma \ref{l5.3}, we conclude that
\begin{equation*}
\lf[\int_{B_\boz(x_0,s)\cap\boz_{2\dz s}}|v|^2\,dx\r]^{1/2}\ls
\dz s\lf(\int_{W_{s,2\dz}}|\nabla v|^2\,dx\r)^{1/2}
\ls \dz s(\dz s^n)^{\epsilon}\lf[\int_{B_\boz(x_0,2s)}|\nabla v|^p\,dx\r]^{1/p},
\end{equation*}
where $W_{s,2\dz}$ is as in \eqref{5.3},
$p\in(2,\fz)$ as in Lemma \ref{l4.1}, and $\epsilon:=\frac{1}{2}-\frac{1}{p}$, which,
combined with $|\nabla \tz_{\dz s}|\ls(\dz s)^{-1}$ and the H\"older inequality, further implies that
\begin{align}\label{5.7}
&\lf|\int_{B_\boz(x_0,s)}\nabla(\tz_{\dz s}v)\cdot\nabla(w-v)\,dx\r|\\ \nonumber
&\hs\le\lf|\int_{B_\boz(x_0,s)}v\nabla\tz_{\dz s}\cdot\nabla(w-v)\,dx\r|
+\lf|\int_{B_\boz(x_0,s)}\tz_{\dz s}\nabla v\cdot\nabla(w-v)\,dx\r|\\ \nonumber
&\hs\ls\lf[(\dz s)^{-1}\|v\|_{L^2(B_\boz(x_0,s)\cap\boz_{2\dz s})}
+\|\nabla v\|_{L^2(B_\boz(x_0,s)\cap\boz_{2\dz s};\rn)}\r]
\lf\|\nabla(w-v)\r\|_{L^2(B_\boz(x_0,s);\rn)}\\ \nonumber
&\hs\ls(\dz s^n)^{\epsilon}
\|\nabla v\|_{L^{p}(B_\boz(x_0,2s);\rn)}\|\nabla(w-v)\|_{L^2(B_\boz(x_0,s);\rn)}.
\end{align}
Furthermore, similarly to \eqref{5.7}, we have
\begin{align}\label{5.8}
&\lf|\int_{B_\boz(x_0,s)}\nabla((1-\tz_{\dz s})v)\cdot\nabla(w-v)\,dx\r|\\ \nonumber
&\hs\ls(\dz s^n)^{\epsilon}
\|\nabla v\|_{L^{p}(B_\boz(x_0,2s);\rn)}\|\nabla(w-v)\|_{L^2(B_\boz(x_0,s);\rn)}.
\end{align}
Thus, from \eqref{5.6}, \eqref{5.7}, and \eqref{5.8}, it follows that
$$\|\nabla(w-\wz{v})\|_{L^2(V_s;\rn)}^2\ls(\dz s^n)^{\epsilon}
\|\nabla v\|_{L^{p}(B_\boz(x_0,2s);\rn)}\|\nabla(w-v)\|_{L^2(B_\boz(x_0,s);\rn)},
$$
which further implies that
\begin{equation}\label{5.9}
\|\nabla(w-v)\|_{L^2(B_\boz(x_0,s);\rn)}\ls(\dz s^n)^{\epsilon}
\|\nabla v\|_{L^{p}(B_\boz(x_0,2s);\rn)}.
\end{equation}
Then, by \eqref{5.9}, Lemma \ref{l4.1}, and $\epsilon:=\frac{1}{2}-\frac{1}{p}$, we find that
$$\lf[\fint_{B_\boz(x_0,s)}|\nabla(v-w)|^2\,dx\r]^{1/2}\ls\dz^{\epsilon}
\lf[\fint_{B_\boz(x_0,2s)}|\nabla v|^p\,dx\r]^{1/p}
\ls\dz^{\epsilon}
\lf[\fint_{B_\boz(x_0,4s)}|\nabla v|^2\,dx\r]^{1/2}.$$
This finishes the proof of \eqref{5.5} and hence of Lemma \ref{l5.1}.
\end{proof}

Now, we prove Theorem \ref{t1.3} via using Lemma \ref{l5.1} and Theorem \ref{t1.1}.

\begin{proof}[Proof of Theorem \ref{t1.3}]
We first show (i). Via replacing Lemma \ref{l4.3} by Lemma \ref{l5.1}, and
repeating the proof of Theorem \ref{t1.2}(i),
we can prove (i). We omit the details here.

Next, we show (ii). Let $p\in(1,\fz)$ and $\omega\in A_p(\rn)$.
Assume that $u$ is the weak solution of the Dirichlet problem \eqref{4.12}
with $\mathbf{f}\in L^p_\omega(\boz;\rn)$.

We first assume that $p\in[2,\fz)$. From $\omega\in A_p(\rn)$ and
(i) and (ii) of Lemma \ref{l3.2},
it follows that there exists a sufficiently
large $p_0\in(p,\fz)$ such that
\begin{equation}\label{5.10}
\omega\in A_{\frac p{p'_0}}(\rn)\cap RH_{(\frac{p_0}{p})'}(\rn).
\end{equation}
Let $v$ be as in \eqref{4.14}. Using Lemmas \ref{l5.1} and \ref{l3.3}, and
repeating the proof of \eqref{4.15}, we find that
there exists a positive constant $\dz_0\in(0,1)$,
depending only on $n$, $p$, and $\boz$, such that, if $\boz$ is a
$(\dz,\sz,R)$ quasi-convex domain and $A$ satisfies the $(\dz,R)$-$\mathrm{BMO}$
condition for some $\dz\in(0,\dz_0)$, $\sz\in(0,1)$, and $R\in(0,\fz)$,
or $A\in\mathrm{VMO}(\boz)$, the inequality \eqref{4.15} holds true with
$p$ replaced by $p_0$, namely,
\begin{equation}\label{5.11}
\lf[\fint_{B_\boz(x_0,r)}|\nabla v|^{p_0}\,dx\r]^{\frac1{p_0}}
\ls\lf[\fint_{B_\boz(x_0,2r)}|\nabla v|^{2}\,dx\r]^{\frac1{2}}.
\end{equation}
Then, by \eqref{5.10}, \eqref{5.11}, and Theorem \ref{t1.1}(ii),
we conclude that there exists a positive constant $\dz_0\in(0,1)$,
depending only on $n$, $p$, and $\boz$, such that, if $\boz$ is a
$(\dz,\sz,R)$ quasi-convex domain and $A$ satisfies the $(\dz,R)$-$\mathrm{BMO}$
condition for some $\dz\in(0,\dz_0)$, $\sz\in(0,1)$,
and $R\in(0,\fz)$, or $A\in\mathrm{VMO}(\boz)$, then the weighted Dirichlet problem $(D)_{p,\,\omega}$
is uniquely solvable and \eqref{1.13} holds true in this case.

Now, let $p\in(1,2)$. From this, $\omega\in A_p(\rn)$,
and Lemma \ref{l3.2}(iii), we deduce that
$p'\in(2,\fz)$ and $\omega^{1-p'}\in A_{p'}(\rn)$.
Let $\mathbf{g}\in L^{p'}_{\omega^{1-p'}}(\boz;\rn)$ and $w$ be the weak
solution of the following Dirichlet problem
\begin{equation*}
\begin{cases}
-\mathrm{div}(A^\ast\nabla w)=\mathrm{div}(\mathbf{g})\ \ &\text{in}\ \ \boz,\\
w=0 \ \ &\text{on}\ \ \partial\boz,
\end{cases}
\end{equation*}
where $A^\ast$ denotes the transpose of $A$. By the assumption that $A$ satisfies Assumption \ref{a1}, we find that
$A^\ast$ also satisfies Assumption \ref{a1}. Thus, we know that
there exists a positive constant $\dz_0\in(0,1)$,
depending only on $n$, $p$, and $\boz$, such that, if $\boz$ is a
$(\dz,\sz,R)$ quasi-convex domain and $A$ satisfies the $(\dz,R)$-$\mathrm{BMO}$
condition for some $\dz\in(0,\dz_0)$, $\sz\in(0,1)$,
and $R\in(0,\fz)$, or $A\in\mathrm{VMO}(\boz)$,  then
\begin{equation}\label{5.12}
\|\nabla w\|_{L^{p'}_{\omega^{1-p'}}(\boz;\rn)}\ls\|\mathbf{g}\|_{L^{p'}_{\omega^{1-p'}}(\boz;\rn)}.
\end{equation}
Using \eqref{5.12} and repeating the proof of \eqref{4.17},
we conclude that, when $p\in(1,2)$ and $\omega\in A_p(\rn)$,
the weighted Dirichlet problem $(D)_{p,\,\omega}$
is uniquely solvable and \eqref{1.13} holds true.
This finishes the proof of (ii) and hence of Theorem \ref{t1.3}.
\end{proof}

To prove Corollary \ref{c1.1} by using Theorem \ref{t1.3}, we need the following
Lemma \ref{l5.4}.

\begin{lemma}\label{l5.4}
Let $n\ge2$ and $\boz\subset\rn$ be a bounded semi-convex domain.
Then there exists a $\dz_0\in(0,1)$ such that $\boz$
is a $(\dz,\sz,R)$ quasi-convex domain for any $\dz\in(0,\dz_0)$,
some $\sz\in(0,1)$, and some $R\in(0,\fz)$.
\end{lemma}

\begin{proof}
Assume that $x\in\partial\boz$ and $\boz$ has the UEBC constant $R_0\in(0,\fz)$.
Then there exists a $\mathbf{v}_x\in S^{n-1}$, depending on $x$, such that
$B(x+R_0\mathbf{v}_x,R_0)\subset\boz^\complement$,
which further implies that, for any $r\in(0,R_0)$,
\begin{equation}\label{5.13}
B(x+r\mathbf{v}_x,r)\subset B(x+R_0\mathbf{v}_x,R_0)\subset\boz^\complement.
\end{equation}
Denote by $L_x$ the $(n-1)$-dimensional plane such that
$L_x$ contains $x$ and has a unit normal direction $\mathbf{v}_x$.
For any given $\dz\in(0,1)$ and any $r\in(0,2R_0\dz)$,
let
$$
H(x,\dz,r):=\{y+t(-\mathbf{v}_x):\ y\in L_x,\ t\in[-\dz r,\fz)\}.
$$
Then, by \eqref{5.13} and a simple geometric observation, we conclude that,
for any given $\dz\in(0,1)$ and any $r\in(0,2R_0\dz)$,
\begin{equation}\label{5.14}
\boz\cap B(x,r)\subset H(x,\dz,r)\cap B(x,r).
\end{equation}
Moreover, from Remark \ref{r2.3}(i), it follows that
$\boz$ is a Lipschitz domain, which, together with the fact that
Lipschitz domains are NTA domains, implies that there exist
a $\sz\in(0,1)$ and an $R_1\in(0,\fz)$ such that,
for any $x\in\partial\boz$ and $r\in(0,R_1)$,
there exists an $x_0\in\boz$, depending on $x$,
such that $B(x_0,\sz r)\subset\boz\cap B(x,r)$.
By this, \eqref{5.14}, and  Remark \ref{r2.2}(iii), we conclude that
there exists a $\dz_0\in(0,1)$ such that
$\boz$ is a $(\dz,\sz,R)$ quasi-convex domain
for any $\dz\in(0,\dz_0)$,  some $\sz\in(0,1)$,
and some $R\in(0,\fz)$.
This finishes the proof of Lemma \ref{l5.4}.
\end{proof}

Next, we show Corollary \ref{c1.1} via using Theorem \ref{t1.3}(ii)
and Lemma \ref{l5.4}.

\begin{proof}[Proof of Corollary \ref{c1.1}]
By Remark \ref{r2.3}(iii), we know that, for any given bounded $C^1$ domain
$\boz$,  $\boz$ is a $(\dz,\sz,R)$
quasi-convex domain for any $\dz\in(0,\dz_0)$, some $\sz\in(0,1)$, and
some $R\in(0,\fz)$, where $\dz_0\in(0,1)$ is a constant depending on $\boz$.
Furthermore, from Lemma \ref{l5.4}, it follows that,
for any bounded semi-convex domain $\boz$, $\boz$ is a $(\dz,\sz,R)$ quasi-convex domain
for any $\dz\in(0,\dz_0)$, some $\sz\in(0,1)$, and some $R\in(0,\fz)$,
where $\dz_0\in(0,1)$ is a constant depending on $\boz$.
Thus, for any given bounded $C^1$ domain or semi-convex domain $\boz$,
$\boz$ is a $(\dz,\sz,R)$ quasi-convex domain
for any $\dz\in(0,\dz_0)$, some $\sz\in(0,1)$, and some $R\in(0,\fz)$,
where $\dz_0\in(0,1)$ is a constant depending on $\boz$. By this
and Theorem \ref{t1.3}(ii), we conclude that
the conclusion of this corollary holds true,
which completes the proof of Corollary \ref{c1.1}.
\end{proof}

\section{Several applications of Theorems \ref{t1.2} and \ref{t1.3}}\label{s6}

In this section, we give several applications of the weighted global
estimates obtained in Theorems \ref{t1.2} and \ref{t1.3}.
More precisely, using Theorems \ref{t1.2}(ii) and \ref{t1.3}(ii),
we obtain the global gradient estimates, respectively, in (weighted) Lorentz spaces,
(Lorentz--)Morrey spaces, (Musielak--)Orlicz spaces (also called generalized Orlicz spaces),
and variable Lebesgue spaces.
We begin with recalling the following notion of the weighted Lorentz space
$L^{q,r}_\omega(\boz)$ on the domain $\boz$.

\begin{definition}\label{d6.1}
Let $n\ge2$ and $\boz$ be a bounded NTA domain in $\rn$.
Assume that $q\in[1,\fz)$, $r\in(0,\fz]$, and $\omega\in A_p(\rn)$ with some $p\in[1,\fz)$.
The \emph{weighted Lorentz space} $L^{q,r}_\omega(\boz)$ is defined by setting
$$L^{q,r}_\omega(\boz):=\lf\{f\ \text{is measurable on}\ \boz:\
\|f\|_{L^{q,r}_\omega(\boz)}<\fz\r\},
$$
where, when $r\in(0,\fz)$,
$$\|f\|_{L^{q,r}_\omega(\boz)}:=\lf\{q\int_0^\fz
\lf[t^q\omega\lf(\lf\{x\in\boz:\ |f(x)|>t\r\}\r)\r]^{r/q}\frac{dt}{t}\r\}^{1/r},
$$
and
$$\|f\|_{L^{q,\fz}_\omega(\boz)}:=\sup_{t\in(0,\fz)}t[\omega\lf(\{x\in\boz:\ |f(x)|>t\}\r)]^{1/q}.
$$

Moreover, the \emph{space} $L^{q,r}_\omega(\boz;\rn)$ is defined via replacing
$L^p_\omega(\boz)$ in \eqref{1.1} by the above $L^{q,r}_\omega(\boz)$ in the
definition of $L^p_\omega(\boz;\rn)$ in \eqref{1.2}.
\end{definition}

It is easy to see that, when $q,\,r\in[1,\fz)$ and $q=r$, $L^{q,r}_\omega(\boz)=L^q_\omega(\boz)$
and $L^{q,r}_\omega(\boz;\rn)=L^q_\omega(\boz;\rn)$.

As applications of both Theorems \ref{t1.2}(ii) and \ref{t1.3}(ii) and the interpolation theorem of operators in the scale
of (weighted) Lorentz spaces, we have the following global gradient estimates
for the Dirichlet problem $(D)_p$ in (weighted) Lorentz spaces.

\begin{theorem}\label{t6.1}
Let $n\ge2$, $\boz\subset\rn$ be a bounded Lipschitz domain, and the matrix $A$ satisfy Assumption \ref{a1}.
Assume that $r\in(0,\fz]$, $\uc_0$ is as in Theorem \ref{t1.2}(i), $p_0:=3+\uc_0$ when $n\ge3$,
or $p_0:=4+\uc_0$ when $n=2$. Then, for any given $p\in(p_0',p_0)$ and any
$\omega\in A_{\frac{p}{p_0'}}(\rn)\cap RH_{(\frac{p_0}{p})'}(\rn)$,
there exists a positive constant $\dz_0\in(0,\fz)$,
depending only on $n$, $p$, the Lipschitz constant of $\boz$, $[\omega]_{A_{\frac{p}{p_0'}}(\rn)}$,
and $[\omega]_{RH_{(\frac{p_0}{p})'}(\rn)}$, such that, if
$A$ satisfies the $(\dz,R)$-$\mathrm{BMO}$ condition for some $\dz\in(0,\dz_0)$
and $R\in(0,\fz)$, or $A\in\mathrm{VMO}(\boz)$,
then, for any weak solution $u\in W^{1,2}_0(\boz)$ of the problem $(D)_2$ with
$\mathbf{f}\in L^{p,r}_\omega(\boz;\rn)$, $\nabla u\in L^{p,r}_\omega(\boz;\rn)$
and
\begin{equation}\label{6.1}
\|\nabla u\|_{L^{p,r}_\omega(\boz;\rn)}\le C\|\mathbf{f}\|_{L^{p,r}_\omega(\boz;\rn)},
\end{equation}
where $C$ is a positive constant depending only on $n$, $p$, $r$,
$[\omega]_{A_{\frac{p}{p_0'}}(\rn)}$, $[\omega]_{RH_{(\frac{p_0}{p})'}(\rn)}$,
and the Lipschitz constant of $\boz$.
\end{theorem}

\begin{proof}
Let $p_0$ be as in Theorem \ref{t6.1}, $p\in(p_0',p_0)$, $r\in(0,\fz]$,
and $\omega\in A_{\frac{p}{p_0'}}(\rn)\cap RH_{(\frac{p_0}{p})'}(\rn)$.
By (i) and (iv) of Lemma \ref{l3.1}, we know that there exists an
$\epsilon\in(0,\min\{p-p_0',p_0-p\})$ such that
$\omega\in A_{\frac{p+\epsilon}{p_0'}}(\rn)\cap RH_{(\frac{p_0}{p+\epsilon})'}(\rn)$
and $\omega\in A_{\frac{p-\epsilon}{p_0'}}(\rn)
\cap RH_{(\frac{p_0}{p-\epsilon})'}(\rn)$.

For any $\mathbf{f}\in L^2(\boz;\rn)$, let
$T:\ \mathbf{f}\mapsto\nabla u_{\mathbf{f}}$,
where $u_{\mathbf{f}}$ is the weak solution of the following Dirichlet problem
\begin{equation}\label{6.2}
\begin{cases}
-\mathrm{div}(A\nabla u_{\mathbf{f}})=\mathrm{div}(\mathbf{f})\ \ &\text{in}\ \ \boz,\\
u_{\mathbf{f}}=0 \ \ &\text{on}\ \ \partial\boz.
\end{cases}
\end{equation}
From \eqref{1.12}, it follows that $T$ is a well-defined linear operator on
both the spaces $L^{p+\epsilon}_\omega(\boz;\rn)$ and $L^{p-\epsilon}_\omega(\boz;\rn)$.
Let $\tz_0:=\frac{p+\epsilon}{2p}$. Then $\tz_0\in(0,1)$ and
$\frac{1}{p}=\frac{1-\tz_0}{p-\epsilon}+\frac{\tz_0}{p+\epsilon}$,
which, together with \eqref{1.12} and the interpolation theorem of operators
on Lorentz spaces (see, for instance, \cite[Theorem 1.4.19]{g14}), further implies that,
for any $\mathbf{f}\in L^{p,r}_\omega(\boz;\rn)$,
\begin{equation*}
\lf\|\nabla u_{\mathbf{f}}\r\|_{L^{p,r}_\omega(\boz;\rn)}=\lf\|T(\mathbf{f})\r\|_{L^{p,r}_\omega(\boz;\rn)}
\ls\lf\|\mathbf{f}\r\|_{L^{p,r}_\omega(\boz;\rn)},
\end{equation*}
where $u_{\mathbf{f}}$ is as in \eqref{6.2}. This finishes the proof of \eqref{6.1} and hence of Theorem \ref{t6.1}.
\end{proof}

\begin{theorem}\label{t6.2}
Let $n\ge2$, $\boz\subset\rn$ be a bounded $\mathrm{NTA}$ domain, $p\in(1,\fz)$, $r\in(0,\fz]$,
and $\omega\in A_p(\rn)$. Assume that the matrix $A$ satisfies Assumption \ref{a1}
and $\boz$ is a $(\dz,\sz,R)$ quasi-convex domain
with some $\dz,\ \sz\in(0,1)$ and $R\in(0,\fz)$.
Then there exists a positive constant $\dz_0\in(0,1)$,
depending only on $n$, $p$, $\boz$, and $[\omega]_{A_p(\rn)}$, such that, if
$\boz$ is a $(\dz,\sz,R)$ quasi-convex domain and $A$ satisfies the $(\dz,R)$-$\mathrm{BMO}$
condition for some $\dz\in(0,\dz_0)$, $\sz\in(0,1)$,
and $R\in(0,\fz)$, or $A\in\mathrm{VMO}(\boz)$, then, for any weak solution $u\in W^{1,2}_0(\boz)$ of the problem
$(D)_2$ with $\mathbf{f}\in L^{p,r}_\omega(\boz;\rn)$, $\nabla u\in L^{p,r}_\omega(\boz;\rn)$ and
\begin{equation}\label{6.3}
\|\nabla u\|_{L^{p,r}_\omega(\boz;\rn)}\le C\|\mathbf{f}\|_{L^{p,r}_\omega(\boz;\rn)},
\end{equation}
where $C$ is a positive constant depending only on $n$, $p$, $r$,
$[\omega]_{A_{p}(\rn)}$, and $\boz$.
\end{theorem}

The proof of Theorem \ref{t6.2} is similar to that of Theorem \ref{t6.1}. We omit the details here.

Next, we recall the definition of the (Lorentz--)Morrey space on the domain $\boz$ as follows.

\begin{definition}\label{d6.2}
Assume that $n\ge2$ and $\boz$ is a bounded NTA domain in $\rn$.
Let $p\in(1,\fz)$, $r\in(0,\fz]$, and $\tz\in[0,n]$.
The \emph{Lorentz--Morrey space} $L^{p,r;\tz}(\boz)$ is defined by setting
$$L^{p,r;\tz}(\boz):=\lf\{f\ \text{is measurable on}\ \boz:\
\|f\|_{L^{p,r;\tz}(\boz)}<\fz\r\},
$$
where
$$\|f\|_{L^{p,r;\tz}(\boz)}:=\sup_{s\in(0,\diam(\boz)]}\sup_{x\in\boz}
\lf\{s^{\frac{\tz-n}{p}}\|f\|_{L^{p,r}(B(x,s)\cap\boz)}\r\}.
$$

Moreover, the \emph{space} $L^{p,r;\tz}(\boz;\rn)$ is defined via replacing
$L^p_\omega(\boz)$ in \eqref{1.1} by the above $L^{p,r;\tz}(\boz)$
in the definition of $L^p_\omega(\boz;\rn)$ in \eqref{1.2}.
\end{definition}

It is worth pointing out that, when $\tz=n$, the Lorentz--Morrey space $L^{p,r;\tz}(\boz)$
is just the Lorentz space; in this case, we denote the spaces $L^{p,r;\tz}(\boz)$
and $L^{p,r;\tz}(\boz;\rn)$ simply, respectively, by $L^{p,r}(\boz)$
and $L^{p,r}(\boz;\rn)$. Moreover, when $p=r$, the space $L^{p,r;\tz}(\boz)$
is just the \emph{Morrey space}; in this case, we denote the spaces $L^{p,r;\tz}(\boz)$
and $L^{p,r;\tz}(\boz;\rn)$ simply by $\cm^{\tz}_p(\boz)$
and $\cm^{\tz}_p(\boz;\rn)$, respectively.

Applying Theorems \ref{t6.1} and \ref{t6.2}, we further
obtain the global gradient estimates for the Dirichlet problem
$(D)_p$ in Lorentz--Morrey spaces as follows.

\begin{theorem}\label{t6.3}
Let $A$ and $\boz$ be as in Theorem \ref{t6.1}, $\uc_0$ as in Theorem \ref{t1.2}(i),
$p_0:=3+\uc_0$ when $n\ge3$, or $p_0:=4+\uc_0$ when $n=2$, $p\in(p_0',p_0)$, $r\in(0,\fz]$,
and $\tz\in(pn/p_0,n]$. Then there exists a positive constant $\dz_0\in(0,\fz)$,
depending on $n$, $p$, $r$, $\tz$, and $\boz$, such that, if
$A$ satisfies the $(\dz,R)$-$\mathrm{BMO}$
condition for some $\dz\in(0,\dz_0)$ and $R\in(0,\fz)$, or $A\in\mathrm{VMO}(\boz)$,
then there exists a positive constant $C$, depending only on $n$, $p$, $r$, $\tz$, and
the Lipschitz constant of $\boz$, such that,
for any weak solution $u\in W^{1,2}_0(\boz)$ of the problem $(D)_2$
with $\mathbf{f}\in L^{p,r;\tz}(\boz;\rn)$,
$\nabla u\in L^{p,r;\tz}(\boz;\rn)$ and
\begin{equation}\label{6.4}
\|\nabla u\|_{L^{p,r;\tz}(\boz;\rn)}\le C\|\mathbf{f}\|_{L^{p,r;\tz}(\boz;\rn)}.
\end{equation}
\end{theorem}

\begin{theorem}\label{t6.4}
Let $n\ge2$, $\boz\subset\rn$ be a bounded $\mathrm{NTA}$ domain, $p\in(1,\fz)$, $r\in(0,\fz]$, and $\tz\in(0,n]$.
Assume that the matrix $A$ satisfies Assumption \ref{a1} and $\boz$ is a $(\dz,\sz,R)$ quasi-convex domain
with some $\dz,\ \sz\in(0,1)$ and $R\in(0,\fz)$.
Then there exists a positive constant $\dz_0\in(0,1)$,
depending only on $n$, $p$, $r$, $\tz$, and $\boz$,  such that, if
$\boz$ is a $(\dz,\sz,R)$ quasi-convex domain and $A$ satisfies the $(\dz,R)$-$\mathrm{BMO}$
condition for some $\dz\in(0,\dz_0)$, $\sz\in(0,1)$, and $R\in(0,\fz)$,
or $A\in\mathrm{VMO}(\boz)$, then, for any weak solution $u\in W^{1,2}_0(\boz)$ of the problem
$(D)_2$ with $\mathbf{f}\in L^{p,r;\tz}(\boz;\rn)$, $\nabla u\in L^{p,r;\tz}(\boz;\rn)$ and
\begin{equation}\label{6.5}
\|\nabla u\|_{L^{p,r;\tz}(\boz;\rn)}\le C\|\mathbf{f}\|_{L^{p,r;\tz}(\boz;\rn)},
\end{equation}
where $C$ is a positive constant depending only on $n$, $p$, $r$,
$\tz$, and $\boz$.
\end{theorem}

As corollaries of Theorems \ref{t6.3} and \ref{t6.4}, we have the following
global gradient estimates in Morrey spaces.

\begin{corollary}\label{c6.1}
Let $\boz\subset\rn$ be a bounded $\mathrm{NTA}$ domain and $A$ satisfy Assumption \ref{a1}.
\begin{itemize}
\item[\rm(i)] Assume further that $\boz$ is a bounded Lipschitz domain, $\uc_0$ is as in Theorem \ref{t1.2}(i),
$p_0:=3+\uc_0$ when $n\ge3$, or $p_0:=4+\uc_0$ when $n=2$, $p\in(p_0',p_0)$, and $\tz\in(pn/p_0,n]$.
Then there exists a positive constant $\dz_0\in(0,\fz)$,
depending on $n$, $p$, $\tz$, and the Lipschitz constant of $\boz$, such that, if
$A$ satisfies the $(\dz,R)$-$\mathrm{BMO}$ condition for some $\dz\in(0,\dz_0)$ and $R\in(0,\fz)$,
or $A\in\mathrm{VMO}(\boz)$,
then there exists a positive constant $C$, depending only on $n$, $p$, $\tz$, and
the Lipschitz constant of $\boz$, such that,
for any weak solution $u\in W^{1,2}_0(\boz)$ of the problem $(D)_2$
with $\mathbf{f}\in \cm^{\tz}_p(\boz;\rn)$,
$\nabla u\in \cm^{\tz}_p(\boz;\rn)$ and
\begin{equation*}
\|\nabla u\|_{\cm^{\tz}_p(\boz;\rn)}\le C\|\mathbf{f}\|_{\cm^{\tz}_p(\boz;\rn)}.
\end{equation*}
\item[\rm(ii)] Let $p\in(1,\fz)$ and $\tz\in(0,n]$.
Then there exists a positive constant $\dz_0\in(0,1)$,
depending only on $n$, $p$, $\tz$, and $\boz$,  such that, if
$\boz$ is a $(\dz,\sz,R)$ quasi-convex domain and $A$ satisfies the $(\dz,R)$-$\mathrm{BMO}$
condition for some $\dz\in(0,\dz_0)$, $\sz\in(0,1)$,
and $R\in(0,\fz)$, or $A\in\mathrm{VMO}(\boz)$,  then, for any weak solution $u\in W^{1,2}_0(\boz)$ of the problem
$(D)_2$ with $\mathbf{f}\in \cm^{\tz}_p(\boz;\rn)$, $\nabla u\in \cm^{\tz}_p(\boz;\rn)$ and
\begin{equation*}
\|\nabla u\|_{\cm^{\tz}_p(\boz;\rn)}\le C\|\mathbf{f}\|_{\cm^{\tz}_p(\boz;\rn)},
\end{equation*}
where $C$ is a positive constant depending only on $n$, $p$, $\tz$, and $\boz$.
\end{itemize}
\end{corollary}

\begin{remark}\label{r6.1}
Let $\boz\subset\rn$ be a bounded NTA domain and $A:=a+b$ satisfy
Assumption \ref{a1}. For the Dirichlet problem \eqref{1.4},
the estimates \eqref{6.3} and \eqref{6.5} were established in
\cite[Corollary 2.2 and Theorem 2.3]{amp18} under the assumptions that $a$ satisfies the $(\dz,R)$-BMO
condition for some small $\dz\in(0,\fz)$ and some $R\in(0,\fz)$, $b\equiv0$,
and $\boz$ is a bounded Lipschitz domain with small Lipschitz constants.
Thus, the estimates \eqref{6.3} and \eqref{6.5} improve \cite[Corollary 2.2 and Theorem 2.3]{amp18}
via weakening the assumptions on the matrix $A$ and the domain $\boz$.

Moreover, some estimates similar to \eqref{6.3} and \eqref{6.5} for the Dirichlet problem
of some nonlinear elliptic or parabolic equations on Reifenberg flat domains
were obtained in \cite{ap15,bd17a,bd17,mp12,mp11}.
\end{remark}

To show Theorem \ref{t6.3} via using Theorem \ref{t6.1},
we need the following lemma, which is well known (see, for instance,
\cite[Section 7.1.2]{g14} and \cite[Lemma 3.4]{mp12}).

\begin{lemma}\label{l6.1}
\begin{itemize}
\item[{\rm(i)}] Let $s\in[1,\fz)$, $\omega\in A_s(\rn)$, $z\in\rn$, and $k\in(0,\fz)$
be a constant. Assume that $\tau^z(\omega)(\cdot):=\omega(\cdot-z)$ and
$\omega_k:=\min\{\omega,\,k\}$. Then $\tau^z(\omega)\in A_s(\rn)$ and
$[\tau^z(\omega)]_{A_s(\rn)}=[\omega]_{A_s(\rn)}$, and
$\omega_k\in A_s(\rn)$ and $[\omega_k]_{A_s(\rn)}\le c_{(s)}[\omega]_{A_s(\rn)}$,
where $c_{(s)}:=1$ when $s\in[1,2]$, and $c_{(s)}:=2^{s-1}$ when $s\in(2,\fz)$.
\item[{\rm(ii)}] For any $x\in\rn$, let $\omega_\gz(x):=|x|^\gz$,
where $\gz\in\rr$ is a constant. Then, for any given $s\in(1,\fz)$, $\omega_\gz\in A_s(\rn)$
if and only if $\gz\in(-n,n[s-1])$. Moreover, $[\omega_\gz]_{A_s(\rn)}\le C_{(n,\,s,\,\gz)}$,
where $C_{(n,\,s,\,\gz)}$ is a positive constant depending only on $n$, $s$, and $\gz$.
\end{itemize}
\end{lemma}

Now, we show Theorem \ref{t6.3} by using Theorem \ref{t6.1} and Lemma \ref{l6.1}.

\begin{proof}[Proof of Theorem \ref{t6.3}]
We prove this theorem via borrowing some ideas from \cite{mp12,mp11}.
Let $p\in(p_0',p_0)$, $r\in(0,\fz]$, and $\tz\in(pn/p_0,n]$, where $p_0$ is as in Theorem \ref{t6.3}.
Assume that $u$ is the weak solution of the Dirichlet problem \eqref{1.4}
with $\mathbf{f}\in L^{q,r;\tz}(\boz;\rn)$.
For any $x,\,z\in\boz$, $\rho\in(0,\diam(\boz)]$, and $\epsilon\in(0,\tz-\frac{pn}{p_0})$,
let
$$\omega_z(x):=\min\lf\{|x-z|^{-n+\tz-\epsilon},\,\rho^{-n+\tz-\epsilon}\r\}.
$$
Then, by Lemma \ref{l6.1}, we conclude that, for any given $z\in\boz$,
$\omega_z\in A_s(\rn)$ for any given $s\in(1,\fz)$,
and there exists a positive constant $C_{(n,\,s,\,\tz)}$, depending only on
$n$, $s$, and $\tz$, such that $[\omega_z]_{A_s(\rn)}\le C_{(n,\,s,\,\tz)}$.
Moreover, from the assumptions $\tz>np/p_0$ and $\epsilon\in(0,\tz-\frac{pn}{p_0})$,
it follows that $\tz-n-\epsilon>-n/(\frac{p_0}{p})'$. By this, and Lemmas \ref{l3.2}(vi)
and \ref{l6.1}, we conclude that, for any given $z\in\boz$,
$\omega_z\in RH_{(\frac{p_0}{p})'}(\rn)$ and
$[\omega_z]_{RH_{(\frac{p_0}{p})'}(\rn)}\ls1$, which, combined with
Theorem \ref{t6.1} and the assumption that, for any $x\in B(z,\rho)$,
$\omega_z(x)=\rho^{-n+\tz-\epsilon}$, further implies that, for any $z\in\boz$
and $\rho\in(0,\diam(\boz)]$,
\begin{align}\label{6.6}
\|\nabla u\|_{L^{p,r}(B(z,\rho)\cap\boz;\rn)}=\rho^{\frac{n-\tz+\epsilon}{p}}
\|\nabla u\|_{L^{p,r}_{\omega_z}(B(z,\rho)\cap\boz;\rn)}
\ls\rho^{\frac{n-\tz+\epsilon}{p}}\|\mathbf{f}\|_{L^{p,r}_{\omega_z}(\boz;\rn)}.
\end{align}
Moreover, similarly to the proofs of \cite[(5.12) and (5.14)]{mp12},
we know that, for any $z\in\boz$ and $\rho\in(0,\diam(\boz)]$,
$$\|\mathbf{f}\|_{L^{p,r}_{\omega_z}(\boz;\rn)}\ls\|\mathbf{f}\|_{L^{p,r;\tz}(\boz;\rn)}
\rho^{-\frac{\epsilon}{p}},
$$
which, together with \eqref{6.6}, implies that, for any $z\in\boz$ and $\rho\in(0,\diam(\boz)]$,
$$\|\nabla u\|_{L^{p,r}(B(z,\rho)\cap\boz;\rn)}\ls\rho^{\frac{n-\tz}{p}}\|\mathbf{f}\|_{L^{p,r;\tz}(\boz;\rn)}.
$$
From this and the definition of $L^{p,r;\tz}(\boz;\rn)$, we deduce that \eqref{6.4} holds true,
which completes the proof of Theorem \ref{t6.3}.
\end{proof}

\begin{proof}[Proof of Theorem \ref{t6.4}]
The proof of this theorem is similar to that of Theorem \ref{t6.3}. We omit the details here.
\end{proof}

In what follows, a function $f:\ [0,\fz)\to[0,\fz]$ is said to be \emph{almost
increasing} (resp., \emph{almost decreasing}) if there exists a positive constant $L\in[1,\fz)$
such that, for any $s,\,t\in[0,\fz)$ satisfying $s\le t$, $f(s)\le Lf(t)$
[resp., $f(s)\ge Lf(t)$]; in particular, if $L:=1$, then
$f$ is said to be increasing (resp., decreasing). Now, we recall the definitions of weak
$\Phi$-functions and Musielak--Orlicz spaces (also called generalized Orlicz spaces)
as follows (see, for instance, \cite{ch18,ylk17}).
Recall that the \emph{symbol $t\to0^+$} means $t\in(0,\fz)$ and $t\to0$.

\begin{definition}\label{d6.3}
Let $\fai:\,[0,\fz)\to[0,\fz]$ be an increasing function satisfying that
$$\fai(0)
=\lim_{t\to0^+}\fai(t)=0\quad \text{and}\quad \lim_{t\to\fz}\fai(t)=\fz.
$$
\begin{itemize}
\item[{\rm(i)}] Then $\fai$ is called a \emph{weak $\Phi$-function}, denoted by
$\fai\in\Phi_w$, if $t\to\frac{\fai(t)}{t}$ is almost increasing on $(0,\fz)$.
\item[{\rm(ii)}] The \emph{left-continuous generalized inverse} of $\fai$,
denoted by $\fai^{-1}$, is defined by setting, for any $s\in[0,\fz]$,
$$\fai^{-1}(s):=\inf\lf\{t\in[0,\fz):\ \fai(t)\ge s\r\}.$$
\item[{\rm(iii)}] The \emph{conjugate $\Phi$-function} of $\fai$, denoted by $\fai^\ast$,
is defined by setting, for any $t\in[0,\fz)$,
$$\fai^\ast(t):=\sup_{s\in[0,\fz)}\{st-\fai(s)\}.
$$
\item[{\rm(iv)}] Let $E\subset\rn$ be a measurable set. A function
$\fai:\ E\times[0,\fz)\to[0,\fz]$ is called a
\emph{Musielak--Orlicz function} (or a \emph{generalized $\Phi$-function}) on $E$
if it satisfies
\begin{enumerate}
\item[$\mathrm{(iv)_1}$] for any $t\in[0,\fz)$, $\fai(\cdot,t)$ is measurable;
\item[$\mathrm{(iv)_2}$] for almost every $x\in E$,
$\fai(x,\cdot)\in\Phi_w$.
\end{enumerate}
Then the set $\Phi_w(E)$ is defined to be the collection of all
Musielak--Orlicz functions on $E$.
\end{itemize}
\end{definition}

\begin{definition}\label{d6.4}
Let $E\subset\rn$ be a measurable set and $\fai\in\Phi_w(E)$. For any given
$f\in L^1_\loc(E)$, the \emph{Musielak--Orlicz modular} of $f$ is defined by setting
$$\rho_\fai(f):=\int_{E}\fai(x,|f(x)|)\,dx.
$$
Then the \emph{Musielak--Orlicz space} (also called \emph{generalized Orlicz space})
$L^\fai(E)$ is defined by setting
\begin{align*}
&L^\fai(E):=\Big\{u\ \text{is measurable on}\ E:\ \\
&\quad\quad\quad\quad\quad\text{there exists a}\ \lz\in(0,\fz)\
\text{such that}\ \rho_\fai(\lz f)<\fz\Big\}
\end{align*}
equipped with the \emph{Luxembourg} (also called the \emph{Luxembourg--Nakano}) \emph{norm}
\begin{equation*}
\|u\|_{L^\fai(E)}:=\inf\lf\{\lz\in(0,\fz):\
\rho_\fai\lf(\frac{u}{\lz}\r)\le1\r\}.
\end{equation*}
\end{definition}

To obtain the global gradient estimates for the Dirichlet problem
in the scale of Musielak--Orlicz spaces, we need several additional assumptions
for the Musielak--Orlicz function $\fai$.
Let $E\subset\rn$ be a measurable set, $\fai\in\Phi_w(E)$, and $p\in(0,\fz)$.

\medskip

\noindent {\bf Assumption (A0).} There exist positive constants
$\beta\in(0,1)$ and $\gamma\in(0,\fz)$ such that,
for any $x\in E$, $\fai(x,\beta\gamma)\le1\le\fai(x,\gamma)$.

\medskip

\noindent {\bf Assumption (A1).} There exists a $\bz\in(0,1)$
such that, for any $x,\,y\in E$ satisfying $|x-y|\le1$, and any $t\in[1,|x-y|^{-n}]$,
$\bz\fai^{-1}(x,t)\le\fai^{-1}(y,t)$.

\medskip

\noindent {\bf Assumption (A2).} There exist $\bz,\,\sz\in(0,\fz)$
and $h\in L^1(E)\cap L^\fz(E)$ such that, for any $t\in[0,\sz]$ and $x,\ y\in E$,
$$\fai(x,\bz t)\le\fai(y,t)+h(x)+h(y).
$$

\medskip

\noindent {\bf Assumption $\mathrm{\textbf{(aInc)}}_p$.} The function
$s\to \frac{\fai(x,s)}{s^p}$ is almost increasing uniformly in $x\in E$.

\medskip

\noindent {\bf Assumption $\mathrm{\textbf{(aDec)}}_p$.} The function
$s\to \frac{\fai(x,s)}{s^p}$ is almost decreasing uniformly in $x\in E$.

\medskip

Using the weighted global gradient estimates obtained in Theorems \ref{t1.2}(ii)
and \ref{t1.3}(ii), and the limited range extrapolation theorem
established in \cite[Theorem 4.18 and Corollary 4.21]{ch18} in the scale of
Musielak--Orlicz spaces, we obtain the following global gradient estimates in
Musielak--Orlicz spaces for the Dirichlet problem $(D)_p$ on bounded Lipschitz domains
and $(\dz,\sz,R)$ quasi-convex domains.

\begin{theorem}\label{t6.5}
Let $A$ and $\boz$ be as in Theorem \ref{t6.1}, $\uc_0$ as in Theorem \ref{t1.2}(i),
$p_0:=3+\uc_0$ when $n\ge3$, or $p_0:=4+\uc_0$ when $n=2$, and $p_1,\,p_2\in(p_0',p_0)$
with $p_1\le p_2$.
Assume that $\fai\in\Phi_w(\boz)$ satisfies Assumptions $(A0)$ -- $(A2)$,
$\mathrm{(aInc)}_{p_1}$, and $\mathrm{(aDec)}_{p_2}$.
Then there exists a positive constant $\dz_0\in(0,\fz)$,
depending only on $n$, $\fai$, and the Lipschitz constant of $\boz$,  such that, if
$A$ satisfies the $(\dz,R)$-$\mathrm{BMO}$ condition for some $\dz\in(0,\dz_0)$ and $R\in(0,\fz)$,
or $A\in\mathrm{VMO}(\boz)$,
then, for any weak solution $u\in W^{1,2}_0(\boz)$ of the Dirichlet problem $(D)_2$
with $\mathbf{f}\in L^\fai(\boz;\rn)$,
$\nabla u\in L^\fai(\boz;\rn)$ and
\begin{equation*}
\|\nabla u\|_{L^\fai(\boz;\rn)}\le
C\|\mathbf{f}\|_{L^\fai(\boz;\rn)},
\end{equation*}
where $C$ is a positive constant depending only on $n$, $\fai$,
$\diam(\boz)$, and the Lipschitz constant of $\boz$.
\end{theorem}

\begin{theorem}\label{t6.6}
Let $n\ge2$, $\boz\subset\rn$ be a bounded $\mathrm{NTA}$ domain,
and $p_1,\ p_2\in(1,\fz)$ with $p_1\le p_2$.
Assume that the matrix $A$ satisfies Assumption \ref{a1} and $\fai\in\Phi_w(\boz)$
satisfies Assumptions $(A0)$--$(A2)$, $\mathrm{(aInc)}_{p_1}$, and $\mathrm{(aDec)}_{p_2}$.
Then there exists a positive constant $\dz_0\in(0,1)$,
depending only on $n$, $\fai$, and $\boz$,  such that, if
$\boz$ is a $(\dz,\sz,R)$ quasi-convex domain and $A$ satisfies the $(\dz,R)$-$\mathrm{BMO}$
condition for some $\dz\in(0,\dz_0)$, $\sz\in(0,1)$,
and $R\in(0,\fz)$, or $A\in\mathrm{VMO}(\boz)$, then, for any weak solution $u\in W^{1,2}_0(\boz)$ of the problem
$(D)_2$ with $\mathbf{f}\in L^\fai(\boz;\rn)$, $\nabla u\in L^\fai(\boz;\rn)$ and
\begin{equation}\label{6.7}
\|\nabla u\|_{L^\fai(\boz;\rn)}\le
C\|\mathbf{f}\|_{L^\fai(\boz;\rn)},
\end{equation}
where $C$ is a positive constant depending only on $n$, $\fai$, and $\boz$.
\end{theorem}

To prove Theorem \ref{t6.5} via using the Rubio de Francia extrapolation theorem
in the scale of Musielak--Orlicz spaces,
we need the following Lemma \ref{l6.2}, which is just \cite[Corollary 4.21]{ch18}.

\begin{lemma}\label{l6.2}
Let $n\ge2$, $\boz\subset\rn$ be a bounded $\mathrm{NTA}$ domain, $f$ and $h$ be
two given non-negative measurable functions on $\boz$,  and $1<p_1<p<p_2<\fz$.
Assume that, for any given $\omega\in A_{p/p_1}(\rn)\cap RH_{(p_2/p)'}(\rn)$,
\begin{equation*}
\|f\|_{L^p_\omega(\boz)}\le C\|h\|_{L^p_\omega(\boz)},
\end{equation*}
where $C$ is a positive constant depending only
on $n$, $p$, $\boz$, $[\omega]_{A_{p/p_1}(\rn)}$,
and $[\omega]_{RH_{(p_2/p)'}(\rn)}$.
If $\fai\in\Phi_w(\boz)$ satisfies Assumptions $(A0)$--$(A2)$,
$\mathrm{(aInc)}_{q_1}$, and $\mathrm{(aDec)}_{q_2}$ for some $p_1<q_1\le q_2<p_2$,
then there exists a positive constant $C$,
depending only on $n$, $\boz$, and $\fai$, such that
$\|f\|_{L^{\fai}(\boz)}\le C\|h\|_{L^{\fai}(\boz)}$.
\end{lemma}

Now, we show Theorem \ref{t6.5} via using Theorem \ref{t1.2} and Lemma \ref{l6.2}.

\begin{proof}[Proof of Theorem \ref{t6.5}]
Let $u$ be the weak solution of the Dirichlet problem
\begin{equation*}
\begin{cases}
-\mathrm{div}(A\nabla u)=\mathrm{div}(\mathbf{f})\ \ &\text{in}\ \ \boz,\\
u=0 \ \ &\text{on}\ \ \partial\boz.
\end{cases}
\end{equation*}
Assume that $p\in(p_0',p_0)$,
where $p_0$ is as in Theorem \ref{t6.1}.
Then, by Theorem \ref{t1.2}, we conclude that,
for any given $\omega\in A_{\frac{p}{p_0'}}(\rn)\cap
RH_{(\frac{p_0}{p})'}(\rn)$, there exists a positive constant $\dz_0\in(0,\fz)$,
depending only on $n$, $p$, the Lipschitz constant of $\boz$, $[\omega]_{A_{\frac{p}{p_0'}}(\rn)}$,
and $[\omega]_{RH_{(\frac{p_0}{p})'}(\rn)}$, such that, if
$A$ satisfies the $(\dz,R)$-$\mathrm{BMO}$ condition for some $\dz\in(0,\dz_0)$
and $R\in(0,\fz)$, or $A\in\mathrm{VMO}(\boz)$, then
\begin{equation*}
\|\nabla u\|_{L^p_\omega(\boz;\rn)}\ls\|\mathbf{f}\|_{L^p_\omega(\boz;\rn)}.
\end{equation*}
From this and Lemma \ref{l6.2} with $f:=|\nabla u|$, $h:=|\mathbf{f}|$,
$p_1:=p_0'$, and $p_2:=p_0$,
it follows that
$$\lf\|\nabla u\r\|_{L^{\fai}(\boz;\rn)}
\sim\|f\|_{L^{\fai}(\boz)}
\ls\|h\|_{L^{\fai}(\boz)}\sim\lf\|\mathbf{f}\r\|_{L^{\fai}(\boz;\rn)},
$$
which completes the proof of Theorem \ref{t6.5}.
\end{proof}

\begin{proof}[Proof of Theorem \ref{t6.6}]
The proof of this theorem is similar to that of Theorem \ref{t6.5}. We omit the details here.
\end{proof}

To give more corollaries of Theorems \ref{t6.5} and \ref{t6.6},
we recall some necessary notions for variable exponent functions $p(\cdot)$
as follows (see, for instance, \cite{cf13,dhhr11}).
Let $\mathcal{P}(\rn)$ be the set of all measurable functions
$p:\,\rn\to[1,\fz)$. For any $p\in\mathcal{P}(\rn)$, let
\begin{equation}\label{6.8}
p_+:=\mathop{\mathrm{ess\,sup}}\limits_{x\in\rn}p(x)\ \
\text{and} \ \ p_-:=\mathop{\mathrm{ess\,inf}}\limits_{x\in\rn}p(x).
\end{equation}
Recall that a function $p:\ \rn\to\rr$ is said to satisfy the \emph{local
log-H\"older continuity condition} if there exists a positive constant $C_\loc$ such that, for any
$x,\,y\in\rn$ with $|x-y|\le1/2$,
$$|p(x)-p(y)|\le\frac{C_\loc}{-\log(|x-y|)};
$$
a function $p:\ \rn\to\rr$ is said to satisfy the \emph{log-H\"older decay condition} (at infinity)
if there exist positive constants $C_\fz\in(0,\fz)$ and $p_\fz\in[1,\fz)$ such that, for any
$x\in\rn$,
$$|p(x)-p_\fz|\le\frac{C_\fz}{\log(e+|x|)}.
$$
If a function $p$ satisfies both the local log-H\"older continuity condition
and the log-H\"older decay condition,
then the function $p$ is said to satisfy the \emph{log-H\"older continuity condition}.

Moreover, recall that, for any $\az\in(0,1]$, the \emph{H\"older space} $C^{0,\az}(\boz)$
is defined by setting
$$C^{0,\az}(\boz):=\lf\{g\ \text{is continuous on}\ \boz:\
[g]_{C^{0,\az}(\boz)}:=\sup_{x,\,y\in\boz,\,x\neq y}\frac{|g(x)-g(y)|}{|x-y|^\az}<\fz\r\}.
$$

Then we have the following two corollaries of Theorems \ref{t6.5} and \ref{t6.6}.

\begin{corollary}\label{c6.2}
Assume that $p\in\mathcal{P}(\rn)$ satisfies the log-H\"older continuity condition,
$\uc_0$ is as in Theorem \ref{t1.2}(i),
$p_0:=3+\uc_0$ when $n\ge3$, or $p_0:=4+\uc_0$ when $n=2$,
and $p_0'<p_-\le p_+<p_0$, where $p_-$ and $p_+$ are as in \eqref{6.8}.
Then the conclusion of Theorem \ref{t6.5} holds true if $\fai$
satisfies one of the following cases:
\begin{itemize}
\item[\rm(i)] for any $x\in\boz$ and $t\in[0,\fz)$,
$\fai(x,t):=\phi(t)$, where $\phi\in\Phi_w$ satisfies Assumptions
$\mathrm{(aInc)}_{p_1}$ and $\mathrm{(aDec)}_{p_2}$ with $p_0'<p_1\le p_2<p_0$.
\item[\rm(ii)] for any $x\in\boz$ and $t\in[0,\fz)$,
$\fai(x,t):=a(x)t^{p(x)}$, where $C^{-1}\le a\le C$ with $C$ being a positive constant.
\item[\rm(iii)] for any $x\in\boz$ and $t\in[0,\fz)$, $\fai(x,t):=t^{p(x)}\log(e+t)$.
\item[\rm(iv)] for any $x\in\boz$ and $t\in[0,\fz)$, $\fai(x,t):=t^{p}+a(x)t^q$,
where $p_0'<p<q<p_0$ satisfy $\frac{q}{p}<1+\frac1n$, and $0\le a\in L^\fz(\boz)\cap C^{0,\frac{n}{p}(q-p)}(\boz)$.
\item[\rm(v)] for any $x\in\boz$ and $t\in[0,\fz)$, $\fai(x,t):=t^{p}+a(x)t^{p}\log(e+t)$,
where $p_0'<p<p_0$, and $0\le a\in L^\fz(\boz)$ satisfies the local
log-H\"older continuity condition.
\end{itemize}
\end{corollary}

\begin{corollary}\label{c6.3}
Assume that $p\in\mathcal{P}(\rn)$ satisfies the log-H\"older continuity condition
and $1<p_-\le p_+<\fz$, where $p_-$ and $p_+$ are as in \eqref{6.8}.
Then the conclusion of Theorem \ref{t6.6} holds true if $\fai$
satisfies one of the following cases:
\begin{itemize}
\item[\rm(i)] for any $x\in\boz$ and $t\in[0,\fz)$,
$\fai(x,t):=\phi(t)$, where $\phi\in\Phi_w$ satisfies Assumptions
$\mathrm{(aInc)}_{p_1}$ and $\mathrm{(aDec)}_{p_2}$ with $1<p_1\le p_2<\fz$.
\item[\rm(ii)] for any $x\in\boz$ and $t\in[0,\fz)$,
$\fai(x,t):=a(x)t^{p(x)}$, where $C^{-1}\le a\le C$ with $C$ being a positive constant.
\item[\rm(iii)] for any $x\in\boz$ and $t\in[0,\fz)$, $\fai(x,t):=t^{p(x)}\log(e+t)$.
\item[\rm(iv)] for any $x\in\boz$ and $t\in[0,\fz)$, $\fai(x,t):=t^{p}+a(x)t^q$,
where $1<p<q<\fz$ satisfy $\frac{q}{p}<1+\frac1n$ and $0\le a\in L^\fz(\boz)\cap C^{0,\frac{n}{p}(q-p)}(\boz)$.
\item[\rm(v)] for any $x\in\boz$ and $t\in[0,\fz)$, $\fai(x,t):=t^{p}+a(x)t^{p}\log(e+t)$,
where $p\in(1,\fz)$ and $0\le a\in L^\fz(\boz)$ satisfies the local
log-H\"older continuity condition.
\end{itemize}
\end{corollary}

\begin{remark}\label{r6.2}
Let $n\ge2$, $\boz\subset\rn$ be a bounded NTA domain, and $A$
satisfy Assumption \ref{a1}.
\begin{itemize}
\item[\rm(i)] Assume that, for any $x\in\boz$ and $t\in[0,\fz)$,
$\fai(x,t):=\phi(t)$, where $\phi$ is as in Corollary \ref{c6.3}(i).
For the Dirichlet problem \eqref{1.4},
the estimate \eqref{6.7} in this case was obtained in
\cite[Theorem 3.1]{jlw07} under the assumptions that $a$
satisfies the $(\dz,R)$-BMO condition for some small $\dz\in(0,\fz)$ and some $R\in(0,\fz)$, $b\equiv0$,
and $\boz$ is a bounded Reifenberg flat domain.
Thus, Corollary \ref{c6.3}(i) improves \cite[Theorem 3.1]{jlw07}
via weakening the assumptions on the matrix $A$ and the domain $\boz$.
\item[\rm(ii)] Assume that, for any $x\in\boz$ and $t\in[0,\fz)$,
$\fai(x,t):=t^{p(x)}$, where $p(\cdot)$ is as in Corollary \ref{c6.3}.
For the Dirichlet problem \eqref{1.4},
the estimate \eqref{6.7} in this case was established in \cite[Theorem 2.5]{bow14}
under the assumptions that $a$ has partial
small BMO coefficients, $b\equiv0$, and $\boz$ is a bounded Reifenberg flat domain.

Moreover, a variable exponent type estimate similar to \eqref{6.7}
for the Dirichlet problem of some $p$-Laplace type elliptic equations on Reifenberg
flat domains was also obtained in \cite[Theorem 1.4]{b18}.
\end{itemize}
\end{remark}

\bigskip

\noindent Sibei Yang

\medskip

\noindent School of Mathematics and Statistics, Gansu Key Laboratory of Applied Mathematics
and Complex Systems, Lanzhou University, Lanzhou 730000, People's Republic of China

\smallskip

\noindent{\it E-mail:} \texttt{yangsb@lzu.edu.cn}

\bigskip

\noindent Dachun Yang and Wen Yuan (Corresponding author)

\medskip

\noindent Laboratory of Mathematics and Complex Systems (Ministry of Education of China),
School of Mathematical Sciences, Beijing Normal University, Beijing 100875,
People's Republic of China

\smallskip

\smallskip

\noindent {\it E-mails}: \texttt{dcyang@bnu.edu.cn} (D. Yang)

\noindent\phantom{{\it E-mails:} }\texttt{wenyuan@bnu.edu.cn} (W. Yuan)

\end{document}